\documentclass[11pt,a4paper]{article}

\usepackage[latin1]{inputenc}
\usepackage{amsmath,amssymb,amsthm,srcltx}
\usepackage{mathtools}
\usepackage{cases}
\usepackage{color}
\usepackage[margin]{fixme}
\usepackage{enumerate}
\usepackage{stmaryrd}
\usepackage{natbib}
\usepackage{bbm}
\usepackage{dsfont}
\numberwithin{equation}{section}
\usepackage{fullpage}
\usepackage[normalem]{ulem}
\usepackage{hyperref}

\usepackage{latexsym}
\usepackage{graphicx}

\DeclareSymbolFont{calletters}{OMS}{cmsy}{m}{n}
\DeclareSymbolFontAlphabet{\mathcal}{calletters}

\setcitestyle{numbers,open={[},close={]}}

\def\be{\begin{eqnarray}}
\def\ee{\end{eqnarray}}

\newtheorem{Theorem}{Theorem}[section]
\newtheorem{Definition}[Theorem]{Definition}
\newtheorem{Proposition}[Theorem]{Proposition}

\newtheorem{Assumption}[Theorem]{Assumption}

\newtheorem{Lemma}[Theorem]{Lemma}
\newtheorem{Corollary}[Theorem]{Corollary}
\newtheorem{Remark}[Theorem]{Remark}

\makeatletter \@addtoreset{equation}{section}

\def \E{\mathbb{E}}
\def \F{\mathbb{F}}
\def \L{\mathbb{L}}
\def \P{\mathbb{P}}
\def \Q{\mathbb{Q}}
\def \R{\mathbb{R}}
\def \S{\mathbb{S}}
\def \M{\mathbf{M}}
\def \N{\mathbb{N}}

\def\dbE{\mathbb{E}}

\def\M{\mathbb{M}}

\def\dbP{\mathbb{P}}
\def\dbR{\mathbb{R}}

\def\Uc{{\cal A}}
\def\Bc{{\cal B}}

\def\Ec{{\cal E}}
\def\Fc{{\cal F}}
\def\Gc{{\cal G}}
\def\Hc{{\cal H}}

\def\Lc{{\cal L}}

\def\Pc{{\cal P}}
\def\Qc{{\cal Q}}

\def\Uc{{\cal U}}

\def \Om{\Omega}
\def \om{\omega}

\def \eps{\varepsilon}

\def \0{\mathbf{0}}
\def \H{\mathbb{H}}
\def \Xh{\widehat{X}}

\newcommand{\ba}{\begin{array}}
\newcommand{\ea}{\end{array}}
\newcommand{\bea}{\begin{eqnarray}}
\newcommand{\eea}{\end{eqnarray}}
\newcommand{\beaa}{\begin{eqnarray*}}
\newcommand{\eeaa}{\end{eqnarray*}}
\newcommand{\sint}{\stackrel{\mbox{\tiny$\bullet$}}{}}

\def\qed{ \hfill \vrule width.25cm height.25cm depth0cm\smallskip}

\newcommand{\basa}{\begin{assumption}}
\newcommand{\easa}{\end{assumption}}

\newcommand{\bas}{\begin{assum}}
\newcommand{\eas}{\end{assum}}

\def\esup{\mathop{\rm ess\,sup}}

\def\argmax{\mathop{\rm arg\,max}}

\def\1{{\bf 1}}

\def\MFG{\mathrm{MFG}}

\def\:{\!:\!}

\newtheorem{thm}{Theorem}[section]

\newtheorem{assum}[thm]{Assumption}

\def\x{\times}

\def\1{{\bf 1}}

\def\sigmat{\widetilde{\sigma}}

\def\Ph{\widehat{\P}}

\begin{document}

\title{Entropic optimal planning \\ for path-dependent mean field games}

\author{
	Zhenjie Ren \footnote{CEREMADE, Universit\'e Paris-Dauphine, PSL Research University. ren@ceremade.dauphine.fr.} 
	\and 
	Xiaolu Tan \footnote{Department of Mathematics, The Chinese University of Hong Kong. xiaolu.tan@cuhk.edu.hk. The research of Xiaolu Tan is supported by Hong Kong RGC General Research Fund (project 14302921).}
	\and 
	Nizar Touzi \footnote{CMAP, Ecole Polytechnique, nizar.touzi@polytechnique.edu}
	\and
	Junjian Yang \footnote{FAM, Fakult\"at f\"ur Mathematik und Geoinformation, Vienna University of Technology, A-1040 Vienna, Austria. junjian.yang@tuwien.ac.at}
}
\date{\today}

\maketitle

\begin{abstract}
	In the context of mean field games, with possible control of the diffusion coefficient, we consider a path-dependent version of the planning problem introduced by P.L. Lions: given a pair of marginal distributions $(\mu_0,\mu_1)$, find a specification of the game problem starting from the initial distribution $\mu_0$, and inducing the target distribution $\mu_1$ at the mean field game equilibrium. 
Our main result reduces the path-dependent planning problem into an embedding problem, that is, constructing a McKean-Vlasov dynamics with given marginals $(\mu_0,\mu_1)$.
	Some sufficient conditions on $(\mu_0,\mu_1)$ are provided to guarantee the existence of solutions. 
	We also characterize, up to integrability, the minimum entropy solution of the planning problem.
	In particular, as uniqueness does not hold anymore  in our path-dependent setting, 
	one can naturally introduce an optimal planning problem which would be reduced to an optimal transport problem along controlled McKean-Vlasov dynamics.
\end{abstract}

\vspace{3mm}

\vspace{-0.2cm}\newline

%

\section{Introduction}

During his courses at Coll\`ege de France \cite{Lions2009}, P.-L. Lions introduced the following planning problem for a class mean field games (MFG hereafter): {\it given two marginal distributions $\mu_0$ and $\mu_1$ on $\R^d$, find a solution $(u, m)$ of the following MFG system:}
	 \begin{align}
	   -\partial_t u - \frac{ \sigma^2}{2}\Delta u - H(x, \nabla u) + F(x, m) &= 0, \quad~\, \mbox{in}~(0,1)\times\dbR^d, 
	   \label{eq:HJB_intro} \\
	   \partial_t m - \frac{\sigma^2}{2}  \Delta m + \nabla\cdot\big(m \nabla_zH(x, \nabla u)\big) &= 0, \quad~\, \mbox{in}~(0,1)\times\dbR^d, \label{eq:FP_intro0} \\
	   m(0, \cdot) = \mu_0, \quad m(1, \cdot) &= \mu_1, \quad \mbox{in}~\dbR^d. 
	   \label{eq:FP_intro}
	 \end{align}
{ 
	Namely, let $c(x,b)$ denote the Legendre transform of the Hamiltonian $H(x,z)$ in $z$, 
	then 
	\begin{itemize}
	\item $u$ in \eqref{eq:HJB_intro} corresponds to the value function of the stochastic optimal control problem
	$$
		\sup_{\beta} ~\E \left[ u\big(1, X^{\beta}_1\big) - \int_0^1 \Big( c\big(X^{\beta}_t, \beta_t\big)\Big) + F\big(X^{\beta}_t, m_t\big)  dt \right],
		~~\mbox{subject to}~
		X^{\beta}_t = X_0 + \int_0^t \beta_s ds + W_t,
	$$
	where $W$ is a Brownian motion, and one optimizes over all progressively measurable processes $\beta$;
	\item the Fokker-Plank equation \eqref{eq:FP_intro0} characterizes the marginal distribution of the state process $X^*$ under optimal control $\beta^*_t=\nabla_zH\big(X^*_t,\nabla u(t,X^*_t)\big)$;
	\item and \eqref{eq:FP_intro} collects the initial distribution of the Fokker-Plank equation, as standard, and in addition a final condition which conditions the choice of the final reward function $u(1,.)$
	\end{itemize}
} 
In other wrds, unlike the standard MFG formulation, the HJB equation \eqref{eq:HJB_intro} is not complemented with a terminal condition for $u\big|_{t=1}$, and instead the Fokker-Planck equation \eqref{eq:FP_intro} is equipped with a terminal condition on $m\big|_{t=1}$ in addition to the initial condition $m\big|_{t=0}$. In other words, the planning problem consists in finding an appropriate reward function which stands as the terminal condition for the HJB equation \eqref{eq:HJB_intro}:
	$$
	g:=u\big|_{t=1}.
	$$ 
At the level of the control problem, this can be interpreted as an incentive for the population so that the classical MFG problem has a solution satisfying the marginal constraint $m\big|_{t=0} = \mu_0$ and $m\big|_{t=1} = \mu_1$. For this reason, $g$ is usually referred to as the incentive function. 
	
In the quadratic Hamiltonian setting, Lions \cite{Lions2009} proved an existence and uniqueness result for a large class of initial and target measures. Various extensions have been achieved since then essentially allowing for Hamiltonians with quadratic growth in the gradient, and using weak solutions for the MFG equation, see Achdou, Camilli, and Capuzzo-Dolcetta \cite{ACCD2012}, Porretta \cite{Porretta2014}, Graber, M\'esz\'aros, Silva, and Tonon \cite{GMST2019}, Orrieri, Porretta, and Savar\'e \cite{OPS2019}, Benamou, Carlier, Di Marino, and Nenna \cite{BCDMN2019}, among others.

The main objective of this paper is to extend the formulation of the planning problem to the path-dependent setting. More precisely, the HJB equation is replaced by a possibly path-dependent stochastic control problem, and the Fokker-Planck equation is replaced by the path-dependent stochastic differential equation characterizing the dynamics of the underlying state under the optimal action induced by the control problem. As another extension that we consider in the present paper, we allow for the control of diffusion coefficient which means that, unlike \eqref{eq:HJB_intro}, the HJB equation in the corresponding Markovian setting is allowed to be fully nonlinear.

	By allowing for path dependency, we are considering a much larger class of incentives $\xi$ which may now be chosen as the set of all functionals of the path of the underlying state.
	Therefore, existence of solution should be easier, but we lose the uniqueness feature of the initial planning problem under the monotone condition in Porretta \cite[Theorem 1.3]{Porretta2014}.
	{On the other hand, the multiplicity of solutions in our path-dependent extension raises naturally the planner's optimization problem over all possible $\xi$ according to some performance or loss criterion.
	This point of view is in fact very popular in the literature on contract theory which sets the rules of the so-called delegation problem between a principal and an agent subject to moral hazard. The nature of the incentive salary of the principal to the agent in compensation for the management of some given output is modeled by means of a leader-follower stochastic game: the leader choses the best incentive compensation given the follower's optimal response. The seminal paper by Holstr\"om \& Milgrom \cite{holmstrom1987aggregation} introduces the continuous time modeling of this problem as a Stackelberg stochastic differential game, and obtains the best incentive compensation as a linear function of the output value at the terminal time. The primal inspiration of our results are from Sannikov \cite{Sannikov2008} and Cvitani\'c, Possama\"{i}, and Touzi \cite{CPT2018}, where the optimal incentive contract falls naturally in the more general class of path dependent functions of the output process. We also refer to Elie, Mastrolia \& Possama\"{\i} \cite{EMP2019} for the extension to the multiple agents in Nash equilibrium context. 
	}

Our main results are first stated in the context where the diffusion is not controlled, a similar situation to the semilinear HJB equation \eqref{eq:HJB_intro} in the Markovian case. Under appropriate integrability conditions on the starting and target measures, we provide a complete characterization of the set of all solutions to the path-dependent planning problem in terms of a controlled auxiliary process. Our result then reduces the planning problem into an embedding problem, that is, to find a good controlled McKean-Vlasov dynamic satisfying the marginal constraints. 

Technically, our approach is adapted from the contract theory literature, such as Cvitani\'c, Possama\"{i}, and Touzi \cite{CPT2018}, Elie, Mastrolia, and Possama\"{i} \cite{EMP2019}.
Nevertheless, it consists in a nontrivial adaptation as the drift coefficient of the controlled process is allowed to be unbounded in our setting.
This corresponds to the Markovian case with quadratically growing Hamiltonians in terms of the gradient component in the literature of mean-field planning problem. 

When only the drift of the state process is controlled, we exhibit an explicit planning solution which, up to some integrability requirement, coincides with the (unique) minimum entropy solution of the planning problem. 
When both the drift and the diffusion of the state process are controlled,
{ the distributions of different controlled processes may not be equivalent, which brings some technical difficulties.
By using similar quasi-sure analysis techniques as in the 2nd order BSDE theory by Soner, Touzi and Zhang \cite{STZ2012},
we are still able to}
provide a similar description of the set of all solutions of the planning problem in terms of a controlled auxiliary process. 
This would reduce the optimal planning problem to an optimal transport problem along controlled McKean-Vlasov dynamics.

A remarkable feature of the extension to the controlled diffusion setting is that it encompasses other classes of optimal transport problems, as for instance the martingale optimal transport of Beiglb\"ock, Henry-Labord\`ere, and Penkner \cite{BHLP2013} and Galichon, Henry-Labord\`ere, and Touzi \cite{GHLT2014}, and its connection to the Skorokhod Embedding problem (see Ob\l{}\'oj \cite{Obloj2004} for a review), the martingale Benamou-Brenier problem in Huesmann and Trevisan \cite{HT2019}, Backhoff-Veraguas, Beiglb\"ock, Huesmann, and K\"allblad \cite{BBHK2020}, and the semimartingale optimal transport problem in Mikami and Thieullen \cite{MT2008}, Tan and Touzi \cite{TT2013}, etc.

The rest of the paper is organized as follows. Section \ref{sec:quadratic} provides our minimum entropy solution of the path-dependent planning problem in the purely quadratic setting. This is exactly the path-dependent analogue of \eqref{eq:HJB_intro}. The extension to a larger class of drift control problems is reported in Section \ref{sec:driftcontrol}. Finally, Section \ref{sec:diffusioncontrol} contains our results for the general case when both the drift and the diffusion coefficients are controlled.

	\vspace{0.5em}

	\noindent {\bf Notations}. 
	Denote by $\Om = C([0,1], \R^d)$ the canonical space of all $\R^d$-valued paths on $[0,1]$, equipped with canonical filtration $\F = (\Fc_t)_{t \in [0,1]}$ and canonical process $X$.
	
	Let $\Pc(\R^d)$ be the space of all (Borel) probability measures on $\R^d$, and denote by $\M$  the collection of all flows of probability measure $(m_t)_{t \in [0,1]}$ with $m_t \in \Pc(\R^d)$ for all $t\in[0,1]$. 
	
	Throughout this paper, we fix some initial distribution $\mu_0\in \Pc(\R^d)$, and we denote by $\P_0$ the Wiener measure on $\Om$ with initial distribution $\mu_0$, i.e., $\P_0 \circ X_0^{-1} = \mu_0$ and the process $(X_t -X_0)_{t \in [0,1]}$ is a Brownian motion independent of $X_0$ under $\P_0$.
	
	Finally, for a probability measure $\P$ on $(\Om, \Gc)$, we denote by $\L^p(\P,\Gc)$ the collection of all $\Gc$-measurable random variable with finite $p$-th moment, by $\H^2(\P)$ the collection of all progressively measurable processes $Z$ such that $\E^\P\big[\int_0^1 |Z_t|^2dt\big]<\infty$, and by $\H^2_{\rm loc}(\P)$ the collection of all progressively measurable processes $Z$ such that $\int_0^1 |Z_t|^2dt<\infty$, $\P$-a.s.
	
\section{MFG planning problem: the linear quadratic setting}
\label{sec:quadratic}

	In this section, we introduce a path-dependent version of the Lions' MFG planning problem in the context of the simplest linear quadratic setting,
	and then provide a constructive solution to the planning problem.

\subsection{The path-dependent linear-quadratic MFG problem}	
	
	Recall that $\P_0$ is the Wiener measure on the canonical space $\Om$ with initial distribution $\mu_0$.
	Let $\Pc(\mu_0)$ denote the collection of all (Borel) probability measures $\P$ on $\Om$ equivalent to $\P_0$ with starting measure $\P\circ X_0^{-1}=\mu_0$. 
    For an arbitrary $\P\in\Pc(\mu_0)$, we may find a unique  process $\beta^\P\in\H^2_{\rm  loc}(\P_0)$ such that the density of $\P$ with respect to $\P_0$ has a representation as the Dol\'eans-Dade exponential
	\begin{equation} \label{eq:def_DD}
		\frac{d \P}{d\P_0} 
		=
		\Ec\big(\beta^{\P} \sint X\big)_1
		:=
		\exp\left(\int_0^1\beta^{\P}_s\cdot dX_s - \frac{1}{2}\int_0^1\big|\beta^{\P}_s\big|^2ds\right).
	\end{equation}
	{Indeed, it follows from \cite[Proposition VIII.1.6]{RY1999} that the density process $D^{\dbP}_t:=\frac{d\P}{d\P_0}\big|_{\Fc_t}$, which is a strictly positive continuous martingale, can be represented as 
	  $$ D^{\dbP}_t = \Ec(L^{\dbP})_t := \exp\left(L^{\dbP}_t-\frac{1}{2}\langle L^{\dbP},L^{\dbP}\rangle_t\right), $$
	  with a unique continuous local martingale $L^{\dbP}$ satisfying $L^{\dbP}_0=0$, $\dbP_0$-a.s. 
	By the predictable representation of the Brownian motion, see, e.g., \cite[Theorem V.3.4]{RY1999}, the local martingale $L$ can be represented as a stochastic integral 
	  $$ L^{\dbP}_t = \int_0^t\beta^{\dbP}_s\cdot dX_s. $$
	Therefore, the desired assertion follows.}
 
	It follows from the Girsanov theorem that the canonical process $X$ satisfies the dynamics
	\begin{equation} \label{eq:LQcontrolSDE}
		X_t = X_0 + \int_0^t \beta^{\P}_s ds + W_t^{\P}, \quad  \P\mbox{-a.s.},
	\end{equation}
	for some $\P$-Brownian motion $W^{\P}$. We further define the subspace
	$$
		\Pc_2(\mu_0) 
		  :=
		\left\{
		\P \in \Pc(\mu_0) : \ln\left(\frac{d\P}{d\P_0}\right)\in\L^1(\P_0)
		                              ~\mbox{ and }~
			                      \frac{d\P}{d\P_0}\in\L^2(\P_0)
		\right\}.
	$$
Let  $f: [0,1] \x \Om \x \Pc(\R^d) \longrightarrow \R$ be such that $(t, \om,\mu) \mapsto f_t(\om, \mu)$ is $\F$-progressively measurable for every fixed $ \mu \in \Pc(\R^d)$ and
	$$
		\E^{\P} \left[ \int_0^1 \big| f_t(m_t) \big| dt \right] < \infty,
		\quad \mbox{for all}~\P \in \Pc_2(\mu_0) \mbox{ and }m \in \M.
	$$
Let $\Xi$ define the set of all admissible (path-dependent and measurable) reward function $\xi: \Om \rightarrow \R$ such that $\E^{\P}[\xi^+] < \infty$ for all $\P \in \Pc_2(\mu_0)$.

 \vspace{2mm}

	For $\xi \in \Xi$ and $m\in\M$ with $m_0=\mu_0$, we consider the optimal control problem:
	\be \label{eq:LQcontrolProblem}
		V_0(\xi,m)
		:=
		\sup_{\P \in\Pc_2(\mu_0)} J(\xi,m,\P),
	\ee
	where
	 $$ J(\xi,m,\P)
	 :=
	 \E^{\P} \left[\xi - \int_0^1 \left(\frac12\big|\beta_s^\dbP\big|^2+f_s(m_s)\right)ds \right]. $$

\begin{Remark} \label{rem:controlset}
	{\rm
	An equivalent formulation of the last control problem is to introduce the set of admissible controls 
	\beaa
		\Uc
		 :=
		\Big\{ \beta\in\H^2(\P_0): \E^{\P_0}\big[\Ec(\beta\sint X)\big]=1 
                      ~\mbox{ and }~
                      \Ec\big(\beta \sint X\big)\in\L^2(\P_0)
		\Big\}.
	\eeaa
	Then, each $\beta\in\Uc$ induces a unique equivalent probability measure $\P^\beta$ defined by the density $\frac{d\P^\beta}{d\P_0}=\Ec(\beta\sint X)$, and we have therefore
	\beaa
		V_0(\xi,m)
		:=
		\sup_{\beta\in\Uc} J\big(\xi,m,\P^\beta\big).
	\eeaa
	See Lemma \ref{lem:betaH2} below for the exact correspondence between $\Pc_2(\mu_0)$ and $\Uc$.
}
\end{Remark}

\begin{Definition}[Mean field game] \label{def:MFG}
A probability measure $\Ph \in \Pc_2(\mu_0)$ is a solution of the MFG with reward function $\xi \in \Xi$ if  
 \beaa
   V_0(\xi,  m)=J\big(\xi,  m,\Ph\big)\in\R,
     &\mbox{and}&
   m_t :=
    \Ph \circ X^{-1}_t, 
     ~~\mbox{for all}~ t\in[0,1].
 \eeaa
We denote by $\mathrm{MFG}(\xi,\mu_0)$ the collection of all such solutions of the MFG problem.
\end{Definition}

Our main focus in this paper is on the following mean field game planning problem.

\begin{Definition}[MFG planning] \label{def:MFP}
An admissible reward function $\xi \in \Xi$ is a solution to the MFG planning problem with starting and target distributions $\mu_0, \mu_1 \in \Pc(\R^d)$ if 
 \beaa
 	\Ph \circ X_1^{-1} =\mu_1,
	 &\mbox{for some}&
 	\Ph \in \mathrm{MFG}(\xi,\mu_0).
 \eeaa
We denote by $\mathrm{MFP}(\mu_0,\mu_1)$ the collection of all such solutions of the MFG planning problem.
\end{Definition}	
	
\subsection{Characterization of the solutions of mean field planning problem}

In this section, we provide a characterization of all MFG planning solutions by using a decomposition induced by the dynamic
programming principle. This characterization follows the idea of the reprersentation of the agent problem in the so called Principal-Agent problem as introduced by Sannikov \cite{Sannikov2008}, and further extended in Cvitani\'c, Possama\"{i}, and Touzi \cite{CPT2018}, and Elie, Mastrolia, and Possama\"{i} \cite{EMP2019}. Denote
	\beaa
		\Pc_2(\mu_0, \mu_1) 
		:=
		\big\{\P \in \Pc_2(\mu_0):  \P \circ X_1^{-1} = \mu_1
		\big\},
		&\mu_0,\mu_1\in\Pc(\R^d).
	\eeaa	
	
\begin{Lemma}  \label{lem:betaH2}
	For each $\P\in\Pc_2(\mu_0)$, we have $\beta^\dbP\in\H^2(\dbP_0)$.  
\end{Lemma}

\begin{proof}
	Set $\zeta:=\frac{d\P}{d\P_0}$, and 
	$$
	Y_t:=\ln \E^{\P_0}[\zeta |\Fc_t]
	 = \int_0^t \beta^\P_s \cdot dX_s - \frac{1}{2} \int_0^t \big| \beta^\P_s\big|^2ds, 
	 \quad t\in[0,1],~~\P_0\mbox{-a.s.}	
	$$
	By the Jensen inequality, $Y_t \geq
	  \E^{\P_0}[\ln\zeta|\Fc_t] 
	  \geq 
	  -\E^{\P_0}\big[(\ln\zeta)^-\big|\Fc_t\big]
	  \geq
	  -\E^{\P_0}\big[|\ln\zeta| \big| \Fc_t\big].$
	Then, introducing the stopping times $ \tau_n:=1\wedge\inf\big\{t>0: \E^{\P_0}\big[| \ln\zeta| \big|\Fc_t\big] > n,\,\int_0^t\big| \beta^\P_s\big|^2ds>n\big\},$ $n\in\N,$ 
	 it follows the tower property that 
	\begin{align*}
	\E^{\P_0}\left[\int_0^{\tau_n} \big|\beta^\P_s\big|^2ds\right] 
	 \leq
	2\E^{\P_0}
	\left[\E^{\P_0}\big[ \big| \ln\zeta \big| \big| \Fc_{\tau_n}\big]\right]
	 = 
	2\E^{\P_0}\big[ \big| \ln\zeta \big| \big] 
	 < 
	\infty.
	\end{align*}
	Since $\tau_n\nearrow 1$ as $n\to\infty$, the assertion follows by the monotone convergence theorem. 
	\qed
\end{proof}

	\begin{Theorem} \label{thm:linear_quad}
	For all pair of starting and target measures $(\mu_0, \mu_1) \in \Pc(\R^d)\times \Pc(\R^d)$, we have:
		$$
			\mathrm{MFP}(\mu_0, \mu_1)
			=
			\L^1(\Fc_0, \P_0)
			+
			\left\{\int_0^1 \!\beta^{\P}_t\cdot dX_t - \int_0^1\!\Big(\frac{1}{2}\big|\beta^{\P}_t\big|^2 - f_t\big(\P \circ X_t^{-1}\big)\Big)dt : \P \in \Pc_2(\mu_0, \mu_1)
			\right\}.
		$$
	\end{Theorem}

	\begin{proof}
	``$\supseteq$'': We first prove that $\mathrm{MFP}(\mu_0, \mu_1)$ contains the right hand side set. For arbitrary $Y_0 \in \L^1(\Fc_0, \P_0)$ and $\Ph \in \Pc_2(\mu_0, \mu_1)$, denote $m=(m_t)_{t\in[0,1]}$ with $m_t:=\Ph\circ X_t^{-1}$, and
	 $$ \xi
	     :=
	     Y_0+\int_0^1 \beta^{\Ph}_t\cdot dX_t-\int_0^1\Big(\frac{1}{2}\big|\beta^{\Ph}_t\big|^2 - f_t(m_t)\Big)dt. $$ 
	 Let us verify that $V_0(\xi,m)=J(\xi,m,\Ph)$. This would show that $\Ph \in\mathrm{MFG}(\xi,\mu_0)$ and therefore $\xi\in\mathrm{MFP}(\mu_0, \mu_1)$. 
	
	We directly compute for all $\P\in\Pc_2(\mu_0)$ that
	 \begin{align*}
	 	J(\xi,m,\P) 
	 	 &= \E^{\P}\left[\xi - \int_0^1\Big(\frac12\big|\beta^{\P}_s\big|^2 + f_s\big(m_s\big)\Big)ds\right] \\
	 	 &= \E^{\P_0} \big[Y_0 \big] 
	 	    + \E^{\P}\left[\int_0^1 \beta^{\Ph}_s \cdot  \big(dW^{\P}_s + \beta^{\P}_s ds\big) - \int_0^1 \Big( \frac{1}{2}\big|\beta^{\Ph}_s\big|^2 +\frac12\big|\beta^{\P}_s\big|^2\Big) ds \right].
	 \end{align*}
    We next observe that the stochastic integral above is a true martingale under $\P$, i.e., 
      $$ \E^{\P}\left[\int_0^1 \beta^{\Ph}_s \cdot dW_s^{\P} \right] = 0, $$
     which is due to the following application of the Burkholder-Davis-Gundy inequality together with the Cauchy-Schwarz inequality, and $\beta^{\widehat\P}\in\H(\dbP_0)$ by Lemma \ref{lem:betaH2}:
      \begin{align*}
      	\E^{\P} \left[\sup_{0\leq t\leq 1}\left|\int_0^t \beta^{\Ph}_s\cdot dW_s^{\P} \right|\right]
      	 &\le C_1 \E^{\P}\left[\left(\int_0^1\big| \beta^{\Ph}_s \big|^2 ds \right)^{\frac12} \right] \\
      	 &\le C_1\E^{\P_0}\bigg[\Big(\frac{d\P}{d\P_0}\Big)^2\bigg]^{\frac12} 
      	 \E^{\P_0}\left[\int_0^1\big|\beta^{\Ph}_s\big|^2 ds \right]^{\frac12} 
      	 <\infty.
      \end{align*}
   Then,
    \begin{align*}
    	J\big(\xi,m,\P\big)
    	=
    	\E^{\P_0} \big[Y_0 \big] 
    	-\frac12
    	\E^{\P} \left[\int_0^1\big|\beta^{\P}_s - \beta^{\Ph}_s\big|^2 ds\right],
    \end{align*}
    so that $J\big(\xi,m,\P\big) \le \E^{\P_0} \big[Y_0 \big]$ for all $\P \in \Pc_2(\mu_0)$, and $J(\xi,m,\Ph) = \E^{\P_0} \big[Y_0 \big]$, as required.
	
	\vspace{1em}
	
	\noindent ``$\subseteq$'': Let $\xi \in \mathrm{MFP}(\mu_0, \mu_1)$, with a corresponding MFG solution $\Ph \in \mathrm{MFG}(\xi,\mu_0)$, 
	so that $\Ph \in \Pc_2(\mu_0)$ is solution of the optimal control problem $V_0(\xi, m)$ with $m_t := \Ph \circ X_t^{-1}$ for all $t \in [0,1]$.
	Then, it is clear that $\Ph \in \Pc_2(\mu_0, \mu_1)$.
	We aim to show that one can represent $\xi$ as 
	\begin{equation} \label{eq:rep_xi}
		\xi
		 =
		V_0 + \int_0^1 \beta^{\Ph}_t\cdot dX_t - \int_0^1\Big(\frac{1}{2}\big|\beta^{\Ph}_t\big|^2 - f_t(m_t)\Big)dt,
	\end{equation}
	for some random variable $V_0 \in \L^1(\Fc_0, \P_0)$. To see this, we introduce the process
	 \begin{align*}
	 	V_t 
	 	:= 
	 	\esup_{\P \in \Pc_2(\mu_0)}
	 	\E^{\P} \left[ \xi - \int_t^1 c_s^{\P}ds ~\bigg|~ \Fc_t \right],
	 	\quad \mbox{with }~c_t^\P:=\frac12\big|\beta^{\P}_t\big|^2+f_t(m_t), 
	 	\quad t\in[0,1].
	 \end{align*}
	Then, it is clear that $\E^{\P_0} \big[V_0 \big]  = V_0(\xi, m) < \infty$, so that $V_0\in \L^1(\Fc_0, \P_0)$.
	Moreover, it follows by the dynamic programming principle (see e.g.~Djete, Possama\"i and Tan \cite[Definition 2.1, Remark 2.3 and Theorem 3.1]{DPT2020}) that 
	{
	   \begin{align*}
	   	 V_t = \esup_{\P \in \Pc_2(\mu_0)}\dbE^{\dbP}\left[V_u -\int_t^u c_s^{\P}ds ~\bigg|~ \Fc_t \right], \qquad  u\in[t,1].
	   \end{align*}
	  Moreover, it follows by
	     $$  V_t - \int_0^t c_s^{\P}ds = \esup_{\P \in \Pc_2(\mu_0)}\dbE^{\dbP}\left[V_u -\int_0^u c_s^{\P}ds ~\bigg|~ \Fc_t \right], \qquad  u\in[t,1], $$
	     the following martingale optimal principle:}
	\begin{itemize}
	\item For any $\P \in \Pc_2(\mu_0)$, the process $\big\{V_t - \int_0^t c_s^{\P} ds\big\}_{t\in[0,1]}$ is $\P$-supermartingale. 
    By the Doob-Meyer decomposition together with the predictable representation property of the Brownian motion, we have  
	 \begin{align*}
	 	V_t - \int_0^t c_s^{\P_0} ds 
	 	= V_0 + \int_0^t Z_s\cdot dX_s- A^{\P_0}_t, \quad \P_0\mbox{-a.s.}, 
	 \end{align*}
	 for some $Z\in\H^2_{\rm loc}(\P_0)$ and non-decreasing process $A^{\P_0}$ starting from zero. 
	By the change of measure from $\dbP_0$ to $\dbP$, we have 
	 \begin{align*}
	 	V_t = \int_0^t c_s^{\P} ds + V_0 + \int_0^t Z_s\cdot dW^\P_s - A^\P_t, \quad \P\mbox{-a.s.},
	 \end{align*}
	 with 
	  $$ A^\P_t=A^{\P_0}_t+\int_0^t \big(c_s^{\P}-c_s^{\P_0}-Z_s\cdot\beta^\P_s\big)ds. $$
    Moreover, by uniqueness of the Doob-Meyer decomposition under each $\P$, the processes $A^\P$ are also non-decreasing.
	\item The process $\big\{V_t - \int_0^t c_s^{\Ph} ds\big\}_{t\in[0,1]}$ is a $\Ph$-martingale, i.e.,
	   $$ 0 = A^{\Ph}_t = A^{\P_0}_t+ \int_0^t\big(c_s^{\Ph}-c_s^{\P_0}-Z_s\cdot\beta^{\Ph}_s\big)ds. $$
	  This shows that $A^{\P_0}$ is absolutely continuous with respect to the Lebesgue measure, and provides the expression for the non-decreasing process $A^\P$ which inherits the absolute continuity property with respect to the Lebesgue measure with density:
	\beaa
	\frac{dA^\P_t}{dt}
	  =
	c_t^{\P}+Z_t\cdot\beta^{\Ph}_t
	-c_t^{\Ph}-Z_t\cdot\beta^{\P}_t
	\ge 0,
	\quad \mbox{for all}~~
	\P\in\Pc_2(\mu_0).
	\eeaa
	In particular, $\beta^{\Ph}$ is the maximizer of 
	{the optimization problem $\max_{\beta^\P} \big( Z\cdot\beta^\P-c^\P \big) = \max_{\beta^\P} \big(  Z\cdot\beta^\P-\frac12\big|\beta^\P\big|^2-f \big)$.
	By the first-order condition, we obtain that $\beta^{\Ph}=Z$. }
	\end{itemize}
The previous analysis shows that $\xi$ has the representation
	\begin{eqnarray*}
		\xi =V_1 
		=V_0 
		+ \int_0^1 \Big(\frac12\big|\beta^{\Ph}_s\big|^2+f_s(m_s)\Big) ds 
		+ \int_0^1 \beta^{\Ph}_s\cdot dW^{\Ph}_s,
	\end{eqnarray*}
	which induces \eqref{eq:rep_xi} by substituting $dW^{\Ph}_s=dX_s-\beta^{\Ph}_sds$.
	\qed
	\end{proof}

\begin{Remark}
{
	Theorem \ref{thm:linear_quad} provides a systematic description of the class of all (possibly path-dependent) solutions of the MFG planning problem.
	Nevertheless, it can not be used to give a precise description of the Markovian solution, i.e., $\xi(X_{\cdot}) = u(T, X_T)$ for some function $u$.
	Indeed, to find some $\beta^{\P}$ and initial condition $u(0,X_0)$ such that
	$$
		u(0, X_0) + \int_0^1 \!\beta^{\P}_t\cdot dX_t - \int_0^1\! \frac{1}{2}\big|\beta^{\P}_t\big|^2 dt  = u(T, X_T), ~~\P \mbox{-a.s. (or equivalently}~\P_0 \mbox{-a.s.)},
	$$
	one can formally use It\^o's formula to identify that
	$$
		\beta^{\P}_t = \nabla u (t, X_t), 
		~~\mbox{and}~~
		\frac12 \big| \beta^{\P}_t \big|^2 = \partial_t u(t, X_t) + \frac12 \Delta u(t, X_t),
		~~ d \P \otimes dt \mbox{-a.e.}
	$$
	This reduces exactly to Lions' original PDE formulation of the MFG planning problem in \eqref{eq:HJB_intro}. 
	On the other hand, an advantage of the description in Theorem \ref{thm:linear_quad} of all solutions allows the planner to choose an optimal one, see more discussions in Section \ref{subsec:entropicMFG}.
}
\end{Remark}

\subsection{A constructive solution to the mean field planning problem}	
\label{subsec:existence_linear}
	
	Theorem \ref{thm:linear_quad} provides a characterization of all solutions of the MFG planning problem by means of the probability measures in $\Pc_2(\mu_0, \mu_1)$. We now use this characterization in order to derive an explicit construction of a particular solution.
	
	\vspace{0.5em}
	
	{ Let $\Pc(\R^d \x \R^d)$ denote the set of probability measures on the product space $\R^d\times \R^d$ and let $\Pi(\mu_0, \mu_1)$ denote the set of all probability measures $\pi \in \Pc(\R^d \x \R^d)$} with marginal distributions $\mu_0$ and $\mu_1$, i.e., $\pi(dx, \R^d) = \mu_0(dx)$ and $\pi(\R^d, dy) = \mu_1(dy)$. We say that $\pi$ is a coupling measure between the starting and target probability measures $\mu_0$ and $\mu_1$.
	We also introduce a reference measure 
	$$
		\rho 
		:=
		\P_0 \circ (X_0, X_1)^{-1} 
		\in
		\Pc(\R^d \x \R^d).
	$$ 
Let $\pi\in \Pi(\mu_0, \mu_1)$ be some coupling measure equivalent to the reference measure $\rho$, and consider the corresponding density function $\frac{d\pi}{d\rho}$ on $\R^d\times\R^d$.
We define the following positive random variable on the canonical space $\Omega$ 
	$$
		\zeta
		:=
		\frac{d\pi}{d\rho}(X_0, X_1), 
	$$
 and observe that	
	$$ 
		\E^{\P_0}[\zeta]
		=
		\E^{\P_0}\left[ \frac{d\pi}{d\rho}(X_0, X_1) \right] 
		=
		\int_{\R^d\times\R^d}\frac{d\pi}{d\rho}(x_0,x_1)d\rho(x_0,x_1)
		=
		1.
	$$
By the martingale representation theorem (see, e.g., \cite[Theorem III.4.33]{JS2003}), there exists a $\F$-progressively measurable process $\widehat \beta$ such that 
	$$ 
		M_t
		:= 
		\E^{\P_0} [\zeta|\Fc_t] 
		= 
		M_0 \Ec\big(\widehat{ \beta} \sint X\big)_t
		=
		M_0 \exp\left(\int_0^t\widehat{\beta}_s\cdot dX_s - \frac{1}{2}\int_0^t|\widehat{\beta}_s|^2ds\right). 
	$$
	In particular, as $\pi(dx, \R^d) = \rho(dx, \R^d) = \mu_0(dx)$, we have $M_0 = 1$, $\P_0$-a.s.

Before stating the main result of this section, we recall the notion of entropy of a probability $\Q_1$ with respect to a reference probability $\Q_0$:
 \begin{align*}
  \mathbf{H}\big(\Q_1 | \Q_0\big) := 
  \begin{cases} \displaystyle
    \E^{\Q_1}\left[ \ln\left(\frac{d\Q_1}{d\Q_0} \right)\right]
     =
     \int_\Omega \ln\left(\frac{d\Q_1}{d\Q_0} \right) d\Q_1, & \mbox{whenever }~ \Q_1\ll \Q_0, \\
     \infty, & \mbox{otherwise},
  \end{cases}
 \end{align*}
{ where $\Q_1\ll \Q_0$ means that $\Q_1$ is absolutely continuous with respect to $\Q_0$.}
	
	\begin{Proposition}  \label{prop:MFP-LQ}
		Let $\pi\in \Pi(\mu_0, \mu_1)$ be equivalent to $\rho$, such that the random variable $\zeta:=\frac{d\pi}{d\rho}(X_0,X_1)$ satisfies
		$$ 
			\E^{\P_0}\big[|\ln\zeta|+\zeta^2\big]
			<
			\infty. 
		$$
        Then, the probability measure $\Ph$ defined by $\frac{d\Ph}{d\P_0}=\zeta$ is an element in $\Pc_2(\mu_0,\mu_1)$. 
		Moreover, $\Ph$ is the unique minimizer of $\mathbf{H}(\cdot| \P_0)$ on $\Pc_\pi := \big\{ \P \in \Pc_{2}(\mu_0, \mu_1) :\P \circ (X_0, X_1)^{-1} = \pi \big\}$.
	\end{Proposition}
	
	\begin{proof}
	$\mathrm{(i)}.$ First, by its definition and the transformation formula, it is clear that $\Ph \circ X_0^{-1} = \mu_0$ and $\Ph \circ X_1^{-1} = \mu_1$.
	{
	Indeed, let $B\in\Bc(\dbR^d)$, i.e., Borel-measurable subset in $\dbR^d$, be arbitrary. Then, 
	   \begin{align*}
	   	 \Ph \circ X_0^{-1}(B)
	   	   &= \dbE^{\Ph}\left[\mathbf{1}_{X_0^{-1}(B)}\right] = \dbE^{\dbP_0}\left[\frac{d\pi}{d\rho}(X_0, X_1) \mathbf{1}_{X_0^{-1}(B)}\right] \\
	   	   &= \dbE^{\dbP_0}\left[\frac{d\pi}{d\rho}(X_0, X_1) \mathbf{1}_{X_0^{-1}(B)}\mathbf{1}_{X_1^{-1}(\dbR^d)}\right] \\
	   	   &= \int_{\R^d \x \R^d}\frac{d\pi}{d\rho}(x_0,x_1)\mathbf{1}_{B\times\dbR^d}(x_0,x_1)\rho(x_0,x_1) \\
	   	   &= \pi(B\times\dbR^d) = \mu_0(B).
	   \end{align*}
	   Similarly, one can prove that $\Ph \circ X_1^{-1} = \mu_1$.
    }
	Moreover, by the integrability assumption on $\zeta$, we have 
	  $$ 
	  \E^{\P_0}\bigg[\bigg|\ln\bigg(\frac{d\Ph}{d\P_0}\bigg)\bigg|
	                        +\bigg(\frac{d\Ph}{d\P_0}\bigg)^2
	                 \bigg]=\E^{\P_0}\big[|\ln\zeta|+\zeta^2\big]
			<
			\infty,
	  $$
	and therefore $\Ph \in \Pc_2(\mu_0, \mu_1)$.

		\vspace{0.5em}
		
		\noindent $\mathrm{(ii)}.$
		Let us denote by $K^{\P}(\cdot; x_0, x_1)$ the kernel function of $\P$ conditional on $(X_0, X_1) = (x_0, x_1)$, for any $\P \in \Pc$.
		We observe from the definition of $\Ph$ and Bayes formula that for any $B\in\Bc(\Omega)$  
		 \begin{align*}
		  K^{\Ph}(B;x_0,x_1) 
		   &= \E^{\Ph}\left[\mathbf{1}_{\{X\in B\}}\big| X_0=x_0, X_1=x_1\right] \\
		   &= \frac{\E^{\P_0}\left[\frac{d\pi}{d\rho}(X_0,X_1)\mathbf{1}_{\{X\in B\}}\big| X_0=x_0, X_1=x_1\right]}{\E^{\P_0}\left[\frac{d\pi}{d\rho}(X_0,X_1)\big| X_0=x_0, X_1=x_1\right]} \\
		   &= \frac{\frac{d\pi}{d\rho}(x_0,x_1) \E^{\P_0}\left[\mathbf{1}_{\{X\in B\}}\big| X_0=x_0, X_1=x_1\right]}{\frac{d\pi}{d\rho}(x_0,x_1)} \\
		   &= \E^{\P_0}\left[\mathbf{1}_{\{X\in B\}}\big| X_0=x_0, X_1=x_1\right] 
		    = K^{\P_0}(B;x_0,x_1),
		 \end{align*}
		 as $\frac{d\pi}{d\rho}$ is strictly positive due to the equivalence. 
		Therefore, 
		$$
			K^{\Ph} (\cdot; x_0, x_1) = K^{\P_0}(\cdot; x_0, x_1),
			~\mbox{ for }~\pi\mbox{-a.e.}~(x_0, x_1)\in \R^d \x \R^d.
		$$
		{
		Denote by $C([0,1])|_{x_0,x_1}$ the set of all continuous functions $\omega$ on $[0,1]$ with $\omega(0)=x_0$ and $\omega(1)=x_1$. 
		Denote $\omega^{(-1)}:=\omega|_{(0,1)}$ for $\omega\in C([0,1])$. 
		Further, for any $\P \in \Pc_\pi$, one has 
		  \begin{align*}
		  	\dbP (d \om) &= K^{\P}\big(d \om^{(-1)}; x_0, x_1\big) \pi(d x_0, d x_1), 
		  \end{align*}
		  and
		  \begin{align*}
		  	\dbP_0(d\omega)&=K^{\dbP_0}\big(d\omega^{(-1)};x_0,x_1\big)\rho(dx_0,dx_1).
		  \end{align*}
        By disintegration theorem and non-negativity of the entropy, we have 
         \begin{align*}
         	&\hspace{-2mm}\mathbf{H} \left(\P | \P_0\right) \\
         	&= \int_\Omega \ln\left(\frac{d\dbP}{d\dbP_0}\right) d\dbP \\
         	&= \int_{\R^d \x \R^d}\int_{C([0,1])|_{x_0,x_1}} \ln\left(\frac{K^{\P}\big(d \om^{(-1)}; x_0, x_1\big) \pi(d x_0, d x_1)}{K^{\dbP_0}\big(d\omega^{(-1)};x_0,x_1\big)\rho(dx_0,dx_1)}\right) K^{\P}\big(d \om^{(-1)}; x_0, x_1\big) \pi(d x_0, d x_1)    \\
         	&=  \int_{\R^d \x \R^d}\int_{C([0,1])|_{x_0,x_1}} \ln\left(\frac{K^{\P}\big(d \om^{(-1)}; x_0, x_1\big)}{K^{\dbP_0}\big(d\omega^{(-1)};x_0,x_1\big)}\right) K^{\P}\big(d\om^{(-1)}; x_0, x_1\big) \pi(d x_0, d x_1)  \\
         	&\quad + \int_{\R^d \x \R^d}\int_{C([0,1])|_{x_0,x_1}} \ln\left(\frac{\pi(d x_0, d x_1)}{\rho(dx_0,dx_1)}\right) K^{\P}\big(d \om^{(-1)}; x_0, x_1\big) \pi(d x_0, d x_1)  \\
         	&= \int_{\R^d \x \R^d} \mathbf{H} \Big( K^{\P} \big(\cdot; x_0, x_1 \big) \Big|  K^{\P_0} \big(\cdot; x_0, x_1 \big)  \Big)\pi(dx_0, dx_1) 
         	+
         	\mathbf{H}(\pi | \rho).
         \end{align*} 
        }
		It follows that $\mathbf{H}\big(\Ph \big|\P_0\big) = \mathbf{H}(\pi | \rho)  \le \mathbf{H}(\P| \P_0)$ for all $\P \in \Pc_\pi$.
		
				\vspace{0.5em}
		
		\noindent $\mathrm{(iii)}.$ Finally, the uniqueness follows directly  from the strict convexity of $\P \longmapsto \mathbf{H}( \P| \P_0)$. 
		\qed
	\end{proof}

\subsection{Entropic MFG planning, and further optimal MFG planning solutions}
\label{subsec:entropicMFG}

	Proposition \ref{prop:MFP-LQ} reduces the problem of minimum entropy MFG planning to the standard static Schr\"odinger bridge problem, i.e., minimize the entropy $\mathbf{H}(\pi | \rho)$ among the set $\Pi(\mu_0,\mu_1)$ of all joint measures with marginals $\mu_0$ and $\mu_1$ (see e.g.~the lecture note of Nutz \cite{Nutz2022}).
	However, due to the integrability requirements in Proposition \ref{prop:MFP-LQ}, we need to restrict this set of coupling measures to the following subset 
	\begin{align*}
		\Pi_2(\mu_0, \mu_1)
		:=
		\bigg\{ \pi \in \Pi(\mu_0, \mu_1) : 
			\E^{\P_0}\bigg[ \bigg| \ln\bigg(\frac{d\pi}{d\rho}(X_0,X_1)\bigg) \bigg| \bigg] 
			+
			\E^{\P_0}\bigg[\bigg(\frac{d\pi}{d\rho}(X_0,X_1)\bigg)^2\bigg]
			<
			\infty
		\bigg\}.
	\end{align*}
This set $\Pi_2(\mu_0, \mu_1)$ is convex, but fails to be closed so that the Schr\"odinger bridge problem may not have a solution in $\Pi_2(\mu_0, \mu_1)$. If the solution of the Schr\"odinger bridge problem happens to satisfy the required integrability conditions, then our construction in Proposition \ref{prop:MFP-LQ} provides the minimum entropy solution of the MFG planning problem. This is stated in the following Corollary which is an immediate consequence of our previous results.

\begin{Corollary}\label{Cor:entropicplanning}
Let $\rho = \P_0\circ(X_0,X_1)^{-1}$ be the reference measure on $\R^d\times\R^d$, and assume that the Schr\"odinger bridge problem $\min_{\pi\in\Pi(\mu_0,\mu_1)}\mathbf{H}(\pi | \rho)$ has a unique solution $\pi^*\in\Pi_2(\mu_0, \mu_1)$. Then, the probability measure $\Ph$ defined by    
  $$ \frac{d\Ph}{d\P_0} = \frac{d\pi^*}{d\rho}(X_0,X_1)  $$ 
is the unique minimizer of $\mathbf{H}(\cdot| \P_0)$ on $\Pc_{2}(\mu_0, \mu_1)$.
\end{Corollary}

	\vspace{0.5em}
	
{
	As is standard in the literature, in particular for the Schr\"odinger bridge problem, 
	the relative entropy $H(\P | \P_0)$ can be considered as a distance between the measure $\P$ and the reference measure $\P_0$.
	The above result implies that the corresponding solution induces a population distribution for the output process with smallest departure from the reference distribution $\P_0$, in terms of the entropy.  
}
	
We conclude this section by a formal discussion on the selection among optimal planning solutions. Given our characterization of all solutions to the MFG planning problem in Theorem \ref{thm:linear_quad}, Corollary \ref{Cor:entropicplanning} selects a solution of the MFG planning which has minimum entropy with respect to the Wiener measure. One may consider other optimization criteria which can be seen as the planner problem whose task is to implement the optimal solution of the MFG planning problem in view of some collective objective. Except for the constraint on the target distribution of the population, this point of view is close to the spirit of contract theory in the economics literature, where the planner, called principal, faces a population of agents in Nash equilibrium, see Sannikov \cite{Sannikov2008} and Cvitani\'c, Possama\"{i} and Touzi \cite{CPT2018} for the one-agent setting, and Elie, Mastrolia, and Possama\"{i} \cite{EMP2019} for the corresponding MFG problem.

\section{MFG planning problem under uncontrolled diffusion} \label{sec:driftcontrol}

	In this section, we show that the mean field planning solution of the linear quadratic MFG, as derived in the previous section, can be adapted to a general class of nonlinear MFG problems whose corresponding Hamiltonian has quadratic growth in the gradient.

\subsection{Formulation of the mean field planning problem}

Throughout this section, $U$ is a given closed subset of $\R^d$, and we denote by $\Pc^U_2(\mu_0)$ the subset of all measures $\P\in\Pc_2(\mu_0)$ such that $\beta^\P \in U,$ Leb$\otimes\P$-a.s. 

Let $c: [0,1] \x \Om \x U \x \Pc(\R^d) \longrightarrow \R$ be an $\F$-progressively measurable map with  
$$
		\E^{\P} \left[ \int_0^1 \big| c_s\big(\beta^\P_s, m_s\big) \big| ds \right] < \infty,
		~\mbox{ for all}~m \in\M, ~\P \in \Pc^U_2(\mu_0).
	$$
	
Similarly, we introduce the subset $\Xi^U$ of all measurable reward functions $\xi\in\Xi$ such that $\E^{\P} \big[ \xi^+ \big] < \infty$ for all $\P \in \Pc^U_2(\mu_0)$.

\vspace{2mm}

For all $m\in\M$ and $\xi \in \Xi^U$, we consider the control problem
	\bea  \label{eq:driftcontrolP}
 		V_0(\xi,m)
		:=
		\sup_{\P \in \Pc^U_2(\mu_0)} J \big( \xi, m, \P),
	&\mbox{where}&
		J \big( \xi, m, \P)
		:= 
		\E^{\P}\left[\xi - \int_0^1c_s\big(\beta^\P_s,m_s\big)ds\right]. 
	\eea
The notions of mean field game and mean field planning are defined as in Definitions \ref{def:MFG} and \ref{def:MFP}, up to the substitution of $\Pc_2$ and $\Xi$ by $\Pc^U_2$ and $\Xi^U$.

\subsection{Characterization of the solutions of the mean field planning problem}

We introduce the Hamiltonian $H$ defined on $[0,1] \x \Om \x \R^d \x \Pc(\R^d)$ by
	\bea  \label{eq:def_Hamiltonian}
		H_s(z,m)
		:=
		H_s(\om,z,m)
		:=
		\sup_{b\in U}  \big\{b\cdot z - c_s(\om,b,m) \big\}. 
	\eea
This defines a convex map in $z$. The following Assumption \ref{assum:driftHamiltonian} guarantees that it is finite, so that the supremum in \eqref{eq:def_Hamiltonian} is attained at any point of the partial sub-gradient $\partial_z H_s(z,m)$ of the convex function $H$ in $z$, i.e., 
  $$ \partial_z H_s(z,m) := \left\{y\in\R^d: H_s(z',m)-H_s(z,m)\geq y\cdot(z'-z), \,\forall z'\in\R^d\right\}. $$

Our main result holds on the following condition which restricts the Hamiltonian to have quadratic growth on terms of the gradient.

\begin{Assumption} \label{assum:driftHamiltonian}
The Hamiltonian satisfies the quadratic growth condition: 
\beaa
\mbox{\rm ess}\inf \min_{(s,m)\in[0,1]\x\Pc(\R^d)} \big| \partial_zH_s(z, m) \big| 
\ge C_1|z|-C_2,
&\mbox{for all}&
z\in\R^d,
\eeaa
for some constants $C_1,C_2 > 0$.
\end{Assumption}

{
	Let $\widehat{b}:[0,1] \x \Om \x R^d \x \Pc(\R^d) \longrightarrow \R^d$ be a measurable function such that $\widehat{b}_s(\om, z, m) \in \partial_z H_s(\om, z, m)$ for all $(s,\om,z,m) \in [0,1] \x \Om \x R^d \x \Pc(\R^d)$,
	and  $Z\in\H^2(\P_0)$ be a control process,
	we next consider the controlled McKean-Vlasov SDE 
	\begin{align} \label{eq:MVCSDE}
		X_t = X_0 + \int_0^t \widehat{b}_s\big(Z_s,\P \circ X_s^{-1}\big)ds + W_t^\P, \quad t \in [0,1],\quad 
		\P\mbox{-a.s.,}
	\end{align} 
	where a solution is a probability $\P \in \Pc(\Om)$ on the canonical space $\Om$, such that for some $\P$-Brownian motion $W^\P$ the equality \eqref{eq:MVCSDE} holds.
	Further, let us denote 
}	\begin{align*}
		\mathrm{MKV}(\mu_0, \mu_1) 
		  & := \left\{(Z, \P) \in \H^2(\P_0)\times\Pc^U_2(\mu_0, \mu_1): 
		              \P~\mbox{solution of}~\eqref{eq:MVCSDE} \right\}, \\
		\Xi(\mu_0, \mu_1)
		  & := \L^1(\Fc_0, \P_0)
		  +\left\{Y^Z_1: (Z, \P) \in \mathrm{MKV}(\mu_0, \mu_1)\right\},
	\end{align*}
     with
     $$ Y^Z_t := \int_0^t Z_s \cdot dX_s - \int_0^tH_s\big(Z_s,\P \circ X_s^{-1}\big)ds, \quad t\in[0,1]. $$

\begin{Theorem}
	For all pairs of starting and target measures $(\mu_0, \mu_1) \in \Pc(\R^d)\times \Pc(\R^d)$, we have $\Xi(\mu_0, \mu_1) \subseteq \mathrm{MFP}(\mu_0, \mu_1)$. 
	
	Moreover, the equality $\Xi(\mu_0, \mu_1) = \mathrm{MFP}(\mu_0, \mu_1)$ holds under Assumption \ref{assum:driftHamiltonian}.
\end{Theorem}

\begin{proof}
	$\mathrm{(i)}.$ Let $\xi:=Y_0 + Y_1^Z$ be an arbitrary element in $\Xi(\mu_0, \mu_1)$ with corresponding $(Z, \Ph) \in \mathrm{MKV}(\mu_0, \mu_1)$, and denote $m_s:=\Ph\circ X_s^{-1}$.
	To show that $\xi \in \mathrm{MFP}(\mu_0, \mu_1)$, it is enough to show that $\xi \in \Xi^U$, and moreover, $\Ph$ is a solution of the optimization problem $V_0(\xi, m)$ in \eqref{eq:driftcontrolP}, so that $\Ph \in \MFG(\xi)$.
	
	\vspace{2mm}	

	For an arbitrary $\P \in \Pc^U_2(\mu_0)$, we first check that, by the Burkholder-Davis-Gundy inequality and Cauchy-Schwarz inequality,
	  $$ 
		\E^{\P}\left[\sup_{0\leq t\leq 1}\bigg|\int_0^t {Z}_s\cdot dW^{\P}_s\bigg|\right] 
		 \leq
		C_1\E^{\P_0}\bigg[\bigg(\frac{d\P}{d\P_0}\bigg)^2\bigg]^{1/2} 
		      \E^{\P_0}\bigg[\int_0^1\big| {Z}_s\big|^2 ds\bigg]^{1/2} 
		 <
		\infty.
	 $$
	Then, the stochastic integral $\big(Z\sint W^{\P}\big)$ is a true martingale and therefore $\E^{\P}\big[\big(Z\sint W^{\P}\big)_1\big] = 0$. 
	We compute that
 	\begin{align*}
 		J(\xi,m,\P) 
 	 	&= \E^{\P}\bigg[\xi - \int_0^1 c_s\big(\beta^{\P}_s, {m}_s\big) ds\bigg]  \\
 	 	&= \E^{\P_0} \big[Y_0\big] 
 	      + \E^{\P} \bigg[\int_0^1{Z}_s \cdot \big(dW^{\P}_s + \beta^{\P}_sds\big) 
 	      - \int_0^1 H_s\big({Z}_s,{m}_s\big)ds 
 	      - \int_0^1 c_s\big(\beta^{\P}_s,{m}_s\big) ds\bigg] \\
 	 	&= \E^{\P_0} \big[Y_0\big] 
 	      + \E^{\P}\left[\int_0^1\Big({Z}_s\cdot \beta^{\P}_s - c_s\big(\beta^{\P}_s, {m}_s\big)  - H_s({Z}_s,{m}_s) \Big)ds\right].
 	\end{align*}
	By the definition of the Hamiltonian $H$, it follows that 
	$J\big(\xi,m,\P\big) \le \E^{\P_0} \big[Y_0\big]$ for all $\P \in \Pc^U_2(\mu_0)$.
	{ As $\widehat\dbP$ is the solution to \eqref{eq:MVCSDE} and 
	   $$ \beta_s^{\widehat{\dbP}} = \widehat{b}(Z_s,m_s) \in\partial_zH_s(Z_s,m_s), $$
	   it deduces by \cite[Theorem 23.5]{Rockafellar1970} that $\beta_s^{\widehat{\dbP}}$ is the optimizer of the Hamiltonian, hence
	     $$ Z_s\cdot\beta_s^{\widehat\dbP} -c_s\big(\beta_s^{\widehat\dbP},m_s\big) = H_s(Z_s,m_s), \quad \widehat\dbP\mbox{-a.s.} $$
	Therefore, $J(\xi,m,\Ph) = \E^{\P_0} \big[Y_0\big]$. }

	\vspace{0.5em}

	\noindent $\mathrm{(ii)}.$ Under the additional conditions in Assumption \ref{assum:driftHamiltonian},
	we consider $\xi \in \mathrm{MFP}(\mu_0, \mu_1)$, together with $\Ph \in\MFG(\xi, \mu_0)$ such that $\Ph \circ X_1^{-1}  = \mu_1$
	so that $J(\xi,m,\Ph)=V(\xi,m)$, for  $m_s := \Ph \circ X_s^{-1}$, $s \in [0,1]$.
	Let us define
	\beaa
		V_t 
		 :=
		\esup_{\P \in \Pc^U(\mu_0)}
		\E^{\P} \bigg[\xi-\int_t^1c_s\big(\beta^\P_s,m_s\big)ds\bigg|\,\Fc_t\bigg], 
		~~t\in[0,1].
	\eeaa
	Then, $\E^{\P_0} \big[ V_0 \big] = V_0(\xi, m) \in \R$, so that $V_0 \in \L^1(\Fc_0, \P_0)$.
	Moreover, by the dynamic programming principle, we argue as in Step (ii) of the proof of Theorem \ref{thm:linear_quad} to show the existence of some $Z\in\H^2_{\rm loc}(\P_0)$ such that  
	 \begin{align*}
	 	V_t = V_0 + \int_0^t Z_s\cdot dX_s - \int_0^t \big(Z_s\cdot\beta_s^{\Ph}-c_s^{\Ph}\big) ds,
	 \end{align*}
	 and
	 \begin{align*}
	 	Z_t\cdot\beta^{\Ph}_t-c_t^{\Ph} 
	 	 = \max_{\P\in\Pc^U_2(\mu_0)}\left\{Z_t\cdot\beta^{\P}_t-c_t^{\P} \right\} 
	 	 = H_t(Z_t,m_t).
	 \end{align*}
	Moreover, since $\beta^\P\in\H^2(\P_0)$ by the definition of $\Pc^U_2(\mu_0)$, it follows by Assumption \ref{assum:driftHamiltonian} that $Z \in\H^2(\P_0)$.
	This concludes the proof that $(Z, \Ph) \in \mathrm{MKV}(\mu_0, \mu_1)$, and hence $\xi \in \Xi(\mu_0, \mu_1)$.
	\qed
	\end{proof}

\subsection{Existence of solution to the mean field planning problem}

	Under Assumption \ref{assum:driftHamiltonian}, the last theorem reduces the construction of a solution of the MFG planning problem to the construction of a solution of the McKean-Vlasov SDE \eqref{eq:MVCSDE} with given starting and target marginals. To do this, we adapt the same arguments as in Section \ref{subsec:existence_linear} under the following additional condition.

	\begin{Assumption} \label{assum:H_z_range}
		The Hamiltonian satisfies the full range condition $\partial_zH_t(\om, \R^d, m)=\R^d$ for all $(t,\omega, m)\in[0,1]\times\Omega\x\Pc(\R^d)$.
	\end{Assumption}

	{ The condition in Assumption \ref{assum:H_z_range} ensures that for any $b \in \R^d$, there exits $z \in \R^d$ such that $b \in \partial_z H_t(\om, z, m)$ (or equivalently $b$ is an optimizer in the definition of $H$ in \eqref{eq:def_Hamiltonian}).
	As application, let us consider a probability measure $\pi \in \Pi(\mu_0, \mu_1)$ (i.e. a measure on $\R^d \x \R^d$ with marginals $\mu_0$ and $\mu_1$), assume that $\pi$ is equivalent to the reference measure $\rho := \P_0 \circ (X_0, X_1)^{-1}$ so that one can define $\zeta:=\frac{d\pi}{d\rho}(X_0, X_1)$.
	Recall the Dol\'eans-Dade exponential $\Ec(\cdot) $ defined in \eqref{eq:def_DD},
	we can then choose a progressively measurable process $Z$} satisfying
	\begin{align} \label{def:Z}
		\widehat\beta_s \in \partial_z H_s(Z_s, m_s), ~\mbox{\rm Leb}\otimes\P_0\mbox{-a.s., with $\widehat\beta$ defined by } \zeta =\Ec\big(\widehat\beta\sint X\big)_1.
	\end{align}
	Then following the same argument as in Proposition \ref{prop:MFP-LQ}, one can prove that this provides a solution of the MFG planning problem.
 
	\begin{Proposition} \label{prop:MFP_existence}
	 Let Assumptions \ref{assum:driftHamiltonian} and \ref{assum:H_z_range} hold true, and suppose in addition that there exists $\pi\in\Pi(\mu_0,\mu_1)$ equivalent to the reference measure $\rho=\dbP_0\circ(X_0,X_1)^{-1}$ such that the density $\zeta:=\frac{d\pi}{d\rho}(X_0,X_1)$ satisfies $\E^{\P_0}\big[|\ln\zeta|+\zeta^2\big]<\infty$.
	 Define the measure $\Ph$ equivalent to $\P_0$ by $\frac{d\Ph}{d\P_0}=\zeta$ and let $Z$ be as defined in \eqref{def:Z}. 

	 Then, the pair $(Z, \Ph) \in \mathrm{MKV}(\mu_0, \mu_1)$, and consequently $Y^Z_1\in\mathrm{MFP}(\mu_0, \mu_1)$. 
	 Moreover, $\Ph$ is the unique minimizer of $\mathbf{H}(\cdot| \P_0)$ on $\Pc_\pi := \big\{ \P \in \Pc_{2}(\mu_0, \mu_1) : \P \circ (X_0, X_1)^{-1} = \pi \big\}$.
	\end{Proposition}
	
	\begin{proof}
	  As in Proposition \ref{prop:MFP-LQ}, we have $\Ph \circ X_0^{-1} = \mu_0$, $\Ph \circ X_1^{-1} = \mu_1$. 
	  By Lemma \ref{lem:betaH2}, we obtain $\beta^{\Ph}\in\H^2(\P_0)$. 
	  Due to Assumptions \ref{assum:driftHamiltonian} and \ref{assum:H_z_range}, the process $Z$ defined in \eqref{def:Z} with $m_s=\Ph\circ X_s^{-1}$ satisfies $Z\in\H^2(\P_0)$. 
	  Finally, our construction immediately yields $\Ph \in \Pc^U_2(\mu_0)$ and
	    $$ X_t = X_0 + \int_0^t\widehat{b}_s\big(Z_s,\Ph\circ X_s^{-1}\big)ds + W_t^{\Ph}, \quad  \Ph\mbox{-a.s.}, $$
	    for some measurable selection $\widehat{b}\in\partial_z H$, Leb$\otimes\Ph$-a.s. 
	  Therefore, $\big(Z, \Ph\big) \in \mathrm{MKV}(\mu_0, \mu_1)$ and $Y_1^Z\in \mathrm{MFP}(\mu_0, \mu_1)$.
	  
	  Following the proof of Proposition \ref{prop:MFP-LQ}, we obtain the entropy minimality. 
	\qed
	\end{proof}

\begin{Remark}{\rm
The discussion of Subsection \ref{subsec:entropicMFG} fully applies to the present setting. Therefore, under the appropriate integrability condition on the solution of the static Schr\"odinger bridge problem, Proposition \ref{prop:MFP_existence} provides the minimum entropy solution of the MFG planning problem.
}
\end{Remark}

\section{MFG planning problem under controlled diffusion}
\label{sec:diffusioncontrol}
\subsection{Formulation of the mean field planning problem}

	Let $\S^d$ denote the space of all symmetric matrices, and $\S^d_+$ the subspace of all positive semidefinite symmetric matrices.
	Let $\Pc$ denote the collection of all probability measures $\P$ on the canonical space $\Om$, under which the canonical process $X$ is a diffusion process with the following decomposition
	$$
		X_t = X_0 + \int_0^t \widehat b^{\P}_s ds + \int_0^t \widehat \sigma_s\, dW^{\P}_s, \quad t \in [0,1], \quad \P \mbox{-a.s.},
	$$
for some $\P$-Brownian motion $W^{\P}$.
	We recall from Karandikar \cite[Theorem 3 and below]{Karandikar1995} that the quadratic variation process $\langle X \rangle$ can be defined independently of $\P \in \Pc$,
	so that $\widehat\sigma_t$ can be defined as the unique square root matrix of $\widehat{\sigma}^2_t$ in $\S^d_+$,
	with 
	$$
		\widehat\sigma_t^2
		:= \lim_{\eps \searrow 0} \frac{\langle X \rangle_t - \langle X \rangle_{(t-\eps) \vee 0}}{\eps},
		 \quad t\in[0,1].
	$$	
	Let $U$ be a closed convex subset of $\R^d\times \S^{d}_+$, with the given two marginal distributions $\mu_0$ and $\mu_1$, we will introduce the set $\Pc^U(\mu_0)$ in two different settings:
 	\begin{itemize}
		\item Setting 1: let
		\begin{equation} \label{eq:def_PcU_0}
			\Pc^U (\mu_0) := \left\{\P\in\Pc ~:  \P \circ X_0^{-1} = \mu_0 ~\mbox{and}~ \Big(\widehat b^{\P}_s, \frac12 \widehat\sigma^2_s \Big) \in U, ~\mbox{Leb}\otimes\P\mbox{-a.e.}\right\}.
		\end{equation}
	
		\item Setting 2: let $U$ satisfy
		\begin{equation}\label{assume:kramkov}
			(b, a) \in U 
			~\Longrightarrow~
			(0,a) \in U
			~\mbox{and}~
			b = a^{1/2} \beta
			~\mbox{for some}~\beta \in\R^d,
		\end{equation}
		and define
		$$
			\Qc(\mu_0) := \big\{ \Q \in \Pc : \Q \circ X_0^{-1} = \mu_0,~X ~\mbox{is a}~\Q\mbox{-martingale} \big\},
		$$
		and
		\begin{eqnarray} \label{eq:def_PcU_1}
			\Pc^U(\mu_0)
			&\!\!\!\!:=\!\!\!\!&
			\bigg\{\P\in\Pc : \frac{d\P}{d \Q} = \Ec \bigg( \int_0^\cdot \beta^{\P}_s \cdot  dW^{\Q}_s \bigg)_1, ~\mbox{for some}~\beta^{\P},~\Q\in\Qc(\mu_0),~\mbox{s.t.}~\nonumber \\
			&&~~~~~~~~~~
		              \int_0^1 |\beta^{\P}_s|^2 ds < \infty, ~\Q\mbox{-a.s.}, ~\mbox{and}~\Big(\widehat \sigma_s \beta^{\P}_s, \frac{1}{2}\widehat \sigma^2_s\Big) \in U, ~\mbox{Leb}\otimes\P\mbox{-a.e.}
	  		\bigg\}.~~~~~~
		\end{eqnarray}
	\end{itemize}
With $\Pc^U(\mu_0)$, we introduce
	$$
		\Pc^U (\mu_0, \mu_1) := \big\{
	 	                            \P \in \Pc^U(\mu_0) :
	 	                            \P \circ X_1^{-1} = \mu_1
	 	                       \big\}.
	$$
	We next consider a cost function $c: [0,1] \x \Om \x \R^d \x \S^{d}_+ \x \Pc(\R^d) \longrightarrow  \R \cup \{\infty\}$ such that $(t,\omega)\mapsto c_t(\omega, \cdot)$ is $\F$-progressively measurable, in particular $ c_t(\omega,\cdot) = c_t(\omega_{t\wedge\cdot}, \cdot)$, for all $(t,\omega)\in[0,1]\times \Omega$, and we assume that
	$$
		\E^{\P} \bigg[
			\int_0^1  c^-_s\big(\widehat b^{\P}_s, \widehat\sigma_s^2, m_s\big) ds
		\bigg]
		< \infty,
		~~\mbox{for all}~~\P \in \Pc^U(\mu_0),~~ m\in\mathbb{M},
		~~\mbox{with}~~ c^- (\cdot):= \max(-c (\cdot), 0).
	$$
	We now introduce the following control problem in weak formulation
	\begin{equation}\label{V0:controlvol}
		V_0(\xi, m) := \sup_{\P \in \Pc^U(\mu_0)} J(\xi, m, \P),
		~~\mbox{with}~~
		J(\xi, m, \P) := \E^{\P} \left[ \xi - \int_0^1 c_s\big(\widehat b^{\P}_s, \widehat\sigma_s^2, m_s\big) ds \right],
	\end{equation}
	where the reward function $\xi: \Om \longrightarrow  \R \cup \{ -\infty\}$ is restricted to the set
	$$
		\Xi^U
		 :=
		\left\{ \xi : \Om \rightarrow \R :
			\E^{\P}\big[ \xi^+ \big] < \infty,~\mbox{for all}~\P \in \Pc^U(\mu_0) 
		\right\}.
	$$
\begin{Definition}
		\begin{enumerate}[{\rm (i)}] 
			\item For $\xi \in \Xi^U$ and $\mu_0\in\Pc_2(\R^d)$, we denote by
			       $$
			       \mathrm{MFG}(\xi, \mu_0) 
			        :=
			       \left\{
			       \Ph \in \Pc^U(\mu_0) : J(\xi, m, \Ph) = V_0(\xi, m) \in \R~\mbox{with}~m_s = \Ph \circ X_s^{-1}, ~s \in [0,1]
			       \right\}
			       $$
			       the set of all solutions to the MFG problem with reward function $\xi$.
			\item Given a pair $(\mu_0, \mu_1)$ of starting and target marginals, we denote by
				$$
			          \mathrm{MFP}(\mu_0, \mu_1)
			           :=
			          \left\{
			            \xi \in \Xi^U :
			               \Ph \circ X_1^{-1} = \mu_1 
			                ~\mbox{for some}~\Ph \in \mathrm{MFG}(\xi, \mu_0)
			          \right\}
			       $$
			        the collection of all reward functions $\xi \in \Xi^U$ which induce some MFG solution $\Ph$ with marginals $\Ph\circ X_0^{-1}=\mu_0$, $\Ph\circ X_1^{-1}=\mu_1$.
		\end{enumerate}
	\end{Definition}

\subsection{Characterization of the solutions of the mean field planning problem}

	Let $\Hc^2(\mu_0) := \bigcap_{\P \in \Pc^U(\mu_0)} \H^2(\P)$, where $\H^2(\P)$ denotes the collection of all $\F$-progressively measurable processes $Z: [0,1] \x \Om \longrightarrow \R^d$ such that $\E^{\P} \big[ \int_0^1 |\widehat\sigma_sZ_s|^2 ds \big] < \infty$. The Hamiltonian of the last stochastic control problem is defined by:
	\begin{equation} \label{eqdef:Hamiltonian22}
		H_s(\om, z, \gamma, m)
		:=
		\sup_{(b, a) \in U} \Big\{ b \cdot z + \frac12 a:\gamma - c_s(\om, b, a, m) \Big\}.
	\end{equation}
	Let us denote  the domain of $H$ by
	$$
		D_H(s,\om,m) : = \big\{ (z, \gamma) \in \R^d \x \S^d ~: H_s(\om, z, \gamma, m) < \infty \big\},
	$$
	and by $\partial_{(z,\gamma)} H_s( z, \gamma, m) \subseteq U$ the sub-gradient of the convex function $(z, \gamma) \longmapsto H_s(z, \gamma, m)$ in $U$.
	
	\vspace{0.5em}
	
	Given $\F$-progressively measurable processes $(Z, \Gamma)$ on $\Om$ taking value in $\R^d \x \S^d$, we introduce the McKean-Vlasov SDE
	\begin{eqnarray} \label{eq:MKV_SDE}
		&&X_t = X_0 + \int_0^t \overline b_s \big(Z_s, \Gamma_s, \Ph \circ X_s^{-1} \big) ds 
		                     + \int_0^t \overline\sigma_s \big(Z_s, \Gamma_s, \Ph \circ X_s^{-1}  \big) dW^{\Ph}_s,
		\quad \Ph\mbox{-a.s.}
		\\
		&&\mbox{for some measurable selection}~
		\Big(\overline b_s, \frac12 \overline \sigma_s^2\Big)(z, \gamma, m)
		\in\partial_{(z,\gamma)} H_s( z, \gamma, m) \subseteq U.
		\nonumber
	\end{eqnarray}
	Let $\mathrm{MKV}_0(\mu_0, \mu_1)$ be the collection of all triples $(Z, \Gamma, \Ph)$ such that
	$(Z, \Gamma) \in D_H(\cdot, \widehat m_{\cdot})$ with $\widehat m_s := \Ph \circ X_s^{-1}$,
	and $\Ph \in \Pc^U(\mu_0, \mu_1)$ is a (weak) solution of the last McKean-Vlasov SDE.
	We next define
	$$
		\mathrm{MKV}(\mu_0, \mu_1)
		:=
		\Big\{
			\big(Z, \Gamma, \Ph\big) \in \mathrm{MKV}_0(\mu_0, \mu_1)
			:Z\in\Hc^2(\mu_0)
		\Big\}.
	$$
	Finally, we introduce for all $(Z, \Gamma, \Ph)\in\mathrm{MKV}(\mu_0, \mu_1)$ the $\Fc_1$-measurable random variable
	$$
		Y^{Z, \Gamma, \Ph}_1 
		:=
		\int_0^1 Z_s\cdot dX_s + \int_0^1 \bigg(\frac12 \Gamma_s:\widehat \sigma^2_s - H_s\big( Z_s, \Gamma_s, \Ph\circ X_s^{-1}\big) \bigg) ds.
	$$
\begin{Remark}{\rm 
	We implicitly work here under the ZFC set-theoretic axioms and the continuum hypothesis. 
	Then, for $Z \in \Hc^2(\mu_0)$, the stochastic integral $\int_0^t Z_s \cdot dX_s$ is well-defined under each $\P \in \Pc^U(\mu_0)$ and it can be aggregated as a universal process independent of $\P \in \Pc^U(\mu_0)$, see Nutz \cite[Theorem 2.2, Lemma 2.5]{Nutz2012}.
	Further, for each $\P \in \Pc^U(\mu_0)$, as $\int_0^t Z_s \cdot \widehat{\sigma}_s\beta^{\P}_s ds$ is well defined,
		and 
		$$
			Z_s \cdot \widehat{\sigma}_s\beta^{\P}_s + \frac12 \Gamma_s: \widehat \sigma^2_s - H_s(Z_s, \Gamma_s, \P\circ X_s^{-1}) 
			 \le 
			c_s \big( \widehat{\sigma}_s\beta^{\P}_s, \widehat \sigma^2_s, \P\circ X_s^{-1} \big),
		$$
		it follows that 
		$$
			\int_0^1 \Big(\frac12 \Gamma_s:\widehat \sigma^2_s - H_s( Z_s, \Gamma_s, \Ph\circ X_s^{-1})\Big) ds
		$$
		is pathwisely well-defined under each $\P \in \Pc^U(\mu_0)$.
		Consequently,
		$Y^{Z, \Gamma, \P}_1$ can be aggregated as a universal random variable on $\Om$ taking value in $\R \cup \{-\infty\}$.
}
\end{Remark}

For the main result of this section, we denote $ \pi(U) := \big\{ a : (b,a) \in U ~\mbox{for some}~b \in \R^d \big\},$ and we define $\Lc^1_0(\mu_0):=\bigcap_{\P\in\Pc^U(\mu_0)}\L^1\big(\Fc_{0}^+,\P)$, and  
	 \begin{align*}
	 	\Xi(\mu_0,\mu_1)
		:= 
		\Lc^1_0(\mu_0)
		+\left\{Y_1^{Z,\Gamma,\Ph} : (Z,\Gamma,\Ph) \in \mathrm{MKV}(\mu_0, \mu_1) \right\}.
	 \end{align*}

\begin{Theorem} \label{thm:2MFP_main1}
	   The following holds true:
		\begin{enumerate}[{\rm (i)}] 
			\item In both settings \eqref{eq:def_PcU_0} and \eqref{eq:def_PcU_1} for the definition of $\Pc^U(\mu_0)$, one has 
			     $$ \Xi(\mu_0,\mu_1) \subseteq \mathrm{MFP}(\mu_0, \mu_1). $$
			\item In the setting \eqref{eq:def_PcU_1} for the definition of $\Pc^U(\mu_0)$, let $\xi \in \mathrm{MFP}(\mu_0, \mu_1)$.
			      Assume in addition that $\sup_{\gamma} \big\{ \frac12 a: \gamma - H(\cdot, \gamma) \big\}$ has a maximizer in the domain of $H$ for all $a \in \pi(U)$,
			      and that
				  \begin{equation} \label{eq:kappa_integ}
					\sup_{\P \in \Pc^U(\mu_0)} \E^{\P} \bigg[
					\big| \xi \big|^{2+\kappa} + \int_0^1 \big| c_s \big(\widehat{\sigma}_s\beta^{\P}_s, \widehat\sigma_s^2, m_s\big) \big|^{2+ \kappa} ds
					\bigg] < \infty,
						\quad \mbox{for some}~\kappa > 0.
			      \end{equation}
			      Then, $\xi \in \Xi(\mu_0,\mu_1)$.
				  Moreover, let $\widehat m_s:=\Ph\circ X_s^{-1}$ for some $\Ph\in\mathrm{MFG}(\xi,\mu_0)$, then we may choose the corresponding $Y_0+Y^{Z,\Gamma,\widehat{\P}}_1\in\Xi(\mu_0,\mu_1)$ such that 
				  \begin{equation} \label{eq:equality_optimal_ctrl}
					\argmax_{\P\in\Pc^U(\mu_0)} J(\xi,\widehat m,\P)
						=
					\argmax_{\P\in\Pc^U(\mu_0)} J\Big(Y_0+Y^{Z,\Gamma,\widehat{\P}}_1,\widehat m,\P\Big).
			      \end{equation}
		\end{enumerate}	 
	\end{Theorem}

	\begin{proof}
	(i) In order to prove the inclusion, we only need to verify that, 
	for all $(Z,\Gamma,\Ph) \in \mathrm{MKV}(\mu_0, \mu_1)$ such that $Y_1^{Z,\Gamma,\Ph}\in \Xi(\mu_0,\mu_1)$,
	one has $\Ph\in\mathrm{MFG}\big(Y_1^{Z,\Gamma,\Ph},\mu_0\big)$, i.e.,
	\begin{equation}\label{proof(i)}
		J\left(Y_1^{Z,\Gamma, \Ph},\widehat{m},\Ph\right) = V_0\left(Y_1^{Z,\Gamma,\Ph}, \widehat{m}\right),
		\quad \mbox{where}~\widehat{m}_s:=\Ph\circ X_s^{-1}. 
	\end{equation}
	Indeed, as $J\big(Y_0+Y_1^{Z,\Gamma, \Ph},\widehat{m},\P\big)=\E^{\P_0}\big[Y_0\big]+J\big(Y_1^{Z,\Gamma, \Ph},\widehat{m},\P\big)$, for all $\P\in\Pc^U(\mu_0)$, this implies that $Y_0+Y_1^{Z,\Gamma,\Ph}\in \mathrm{MFP}(\mu_0, \mu_1)$. To prove \eqref{proof(i)}, we first compute for  $\P\in\Pc^U(\mu_0)$ that
	      \begin{align*}
	         J\left(Y_1^{Z,\Gamma,\Ph},\widehat{m}, \P\right) 
	         &= \E^{\P}\left[\int_0^1Z_s\!\cdot\! dX_s 
	                          + \int_0^1\left(\frac12 \Gamma_s\!:\!\widehat\sigma_s^2 
	                                                  - H_s\big(Z_s, \Gamma_s, \widehat{m}_s\big)
	                                                  - c_s\big(\widehat b^{\P}_s, \widehat \sigma_s^2, \widehat{m}_s\big)\right) ds\right] 
	         \\
	         &= \E^{\P}\left[\int_0^1\left(Z_s\cdot \widehat b^{\P}_s + \frac12 \Gamma_s:\widehat\sigma_s^2 -c_s\big(\widehat b^{\P}_s,\widehat\sigma_s,\widehat{m}_s\big) - H_s\big(Z_s, \Gamma_s, \widehat{m}_s\big)\right)ds\right],
	      \end{align*} 
		 as $Z\in\Hc^2(\mu_0)$. 
	Then, it follows from the definition of the Hamiltonian that 
		$ J\big(Y_1^{Z,\Gamma,\Ph}, \widehat{m}, \P\big) \leq 0$ for all $\P\in\Pc^U(\mu_0)$. On the other hand, as $(Z,\Gamma,\Ph)\in \mathrm{MKV}(\mu_0, \mu_1)$, it follows from \eqref{eq:MKV_SDE}
		that $\big(\widehat b^{\Ph}_s,\frac12\widehat{\sigma}^2_s\big)\in\partial_{z,\gamma}H_s(Z_s,\Gamma_s,\widehat{m}_s)$, and therefore the supremum in the Hamiltonian is attained: 
	        $$ Z_s\cdot \widehat b^{\Ph}_s + \frac12 \Gamma_s:\widehat\sigma^2_s -c_s\big(\widehat b^{\Ph}_s,\widehat\sigma^2_s,\widehat{m}_s\big) = H_s\big(Z_s, \Gamma_s, \widehat{m}_s\big), \quad \mathrm{Leb}\otimes\Ph\mbox{-}a.e. $$
	Hence $ J\big(Y_1^{Z,\Gamma,\Ph}, \widehat{m}, \Ph\big) = 0=V_0\big(Y_1^{Z,\Gamma,\Ph}, \widehat{m}\big)$, which concludes the proof of \eqref{proof(i)}.

	\vspace{3mm}

	\noindent (ii)
	Next, let us consider $\xi\in\mathrm{MFP}(\mu_0,\mu_1)$ in the setting \eqref{eq:def_PcU_1} for the definition of $\Pc^U(\mu_0)$, and under the additional conditions in Item (ii) of the statement.
	We observe that by condition \eqref{assume:kramkov}, one has $\Qc(\mu_0) \subseteq \Pc^U(\mu_0)$.
	 
	\vspace{0.5em}

	Let $\Ph\in\mathrm{MFG}(\xi,\mu_0)$ satisfying $\Ph\circ X_1^{-1}=\mu_1$, then $\Ph$ is an optimal control of the stochastic control problem $V_0(\xi,\widehat m)$ defined in \eqref{V0:controlvol} with $\widehat{m}_s:=\Ph\circ X_s^{-1}$, $s\in[0,1]$.

	\vspace{3mm}	

	\noindent (ii-1)
	Let $\Pc^U(t,\omega)$ denote the dynamic version of $\Pc^U(\mu_0)$ by considering the dynamics on $[t,1]$ starting at time $t \in [0,T]$ from the path $\omega\in\Omega$, and consider the dynamic version of the control problem $V_0$: 
	\begin{align*}
		V_t(\omega) 
		 := 
		\sup_{\P\in\Pc^U(t,\omega)}\E^{\P}\left[\xi - \int_t^1 c_s\big( \widehat{\sigma}_s \beta^{\P}_s, \widehat{\sigma}^2_s,\widehat{m}_s\big)ds\right], 
		\quad \mbox{for all}~(t,\omega)\in[0,1]\times\Omega.
	\end{align*}
	Since $\xi$ and $c$ satisfy the addition integrability condition in \eqref{eq:kappa_integ}, 
	it follows that (see e.g.~Soner, Touzi, and Zhang \cite{STZ2012,STZ2013} and Possama\"{i}, Tan, and Zhou \cite{PTZ2018})
	the process $\{V_t,t\in[0,1]\}$ satisfies the dynamic programming principle:
	\begin{equation}\label{DPP}
		V_t
		=
		\sup_{\P\in\Pc^U(t,\omega)}
		\E^{\P}
		\left[ V_{t+h} - \int_t^{t+h} c_s\big( \widehat{\sigma}_s \beta^{\P}_s, \widehat{\sigma}^2_s,\widehat{m}\big)ds \right],
		\quad \P\mbox{-a.s.}, 
	\end{equation}
	for all $\P\in\Pc^U(\mu_0)$ and $0 \le h \le T-t$, 
	so that $V+\int_0^.c_s\big( \widehat{\sigma}_s \beta^{\P}_s, \widehat{\sigma}^2_s,\widehat{m}\big)ds$ is $\P$-supermartingale for all $\P \in \Pc^U(\mu_0)$, and we may introduce the corresponding right-continuous limit $V^+_t (\om) := \lim_{s \searrow t} V_s(\om)$, which inherits the dynamic programming principle \eqref{DPP}. 
	Moreover, one has
	\begin{equation}\label{integV}
		\sup_{\P \in \Pc^U(\mu_0)} 
		\E^{\P} \left[ \sup_{0 \le s \le 1} \big| V^+_s \big|^{2+ \kappa'} \right] < \infty,
		\quad \mbox{for some}~0 < \kappa' <  \kappa.
	\end{equation} 

	\noindent (ii-2) As $\Qc(\mu_0) \subseteq \Pc^U(\mu_0)$, for any $\P \in \Pc^U(\mu_0)$, the process $V^+_t - \int_0^t c_s ( \widehat{\sigma}_s \beta^{\P}_s, \widehat{\sigma}^2_s,\widehat{m})ds$ is a c\`adl\`ag supermartingale under any martingale measure equivalent to $\P$. 
	By the optional decomposition theorem (see El Karoui and Quenez \cite[Theorem 2.4.2]{EKQ1995}, or F\"ollmer and Kramkov \cite{FK1997}), 
	there is a non-decreasing process $K^{\P}$ starting from $K^{\P}_0 = 0$, and a process $Z^{\P}$ such that for $t \in [0,T]$
		$$
			V_t^+ 
			=
			 \xi + \int_t^1 \Big( \widehat \sigma_s \beta^{\P}_s \cdot Z^{\P}_s - c_s\big(  \widehat{\sigma}_s \beta^{\P}_s, \widehat \sigma^2_s,\widehat{m}_s\big) \Big) ds - \int_t^1Z^{\P}_s\cdot dX_s - K^{\P}_t + K^{\P}_T,
			\quad \P\mbox{-a.s.}
		$$
	Let $Z_t :=\frac{d\langle V^+, X \rangle_t}{dt}$, then $Z^{\P} = Z$, $\P$-a.s., for all $\P \in \Pc^U(\mu_0)$,
	and then one can find a process $K$ such that $K^{\P} = K$, $\P$-a.s., for all $\P \in \Pc^U(\mu_0)$.
	Moreover, by standard estimates, it follows from \eqref{integV} that $Z \in \Hc^2(\mu_0)$.
	
	\vspace{0.5em}
	
	Further, since the optimal control $\Ph$ is a maximizer in the dynamic programming principle, we have $K = 0$, $\Ph$-a.s. 
	Then, by the same arguments as in Theorem \ref{thm:linear_quad}, we that
	$$
		V^+_t 
		=
		\xi + \int_t^1 F_s\big(Z_s,\widehat \sigma^2_s, \widehat{m}_s\big)ds - \int_t^1Z_s\cdot dX_s,
		~~t \in [0,T],
		\quad \Ph\mbox{-a.s.},
	$$
	where $F_s(\om, z, a, m) := \sup_{b \in U_a} \big\{ b \cdot z - c_s(\om, b, a, m) \big\}$ with $U_a := \{ b \in \R^d : (b,a) \in U \}$, and
	$$
		F_s\big(Z_s,\widehat \sigma^2_s, \widehat{m}_s\big)
		=
		\widehat \sigma_s \beta^{\Ph}_s \cdot Z_s - c_s\big(\widehat \sigma_s \beta^{\Ph}_s, \widehat \sigma^2_s,\widehat{m}_s\big),
		\quad \mathrm{Leb}\otimes\Ph\mbox{-}a.e.
	$$
	\noindent (ii-3) 
	Notice that
	   \begin{align*}
	   	 & H_s(\om, z, \gamma, m) 
	   	   =
	   	   \sup_{a \in \pi(U)} \bigg\{ \frac12 a: \gamma + F_s(\om, z, a, m) \bigg\}, \\
	   	 & ~~\mbox{with}~~
	   	 \pi(U):= \big\{ a : (b,a)\in U ~\mbox{for some}~b \in \R^d \big\}.
	   \end{align*}
	Then, under the additional conditions in Item (ii) of the present Theorem, 
	we can find a (measurable) process $\Gamma$ in the sub-gradient of $\frac12 a \mapsto c(\cdot, a)$, such that
	$(Z,\Gamma) \in D_H(\cdot, \widehat m_{\cdot})$ and
	$$
		H_s(\cdot,Z_s,\Gamma_s, \widehat m_s)
		 =
		\frac12 \widehat{\sigma}^2_s :\Gamma_s + F_s\big(\cdot,Z_s,\widehat{\sigma}^2_s,\widehat m_s \big),
		\quad \om\mbox{-wisely},
	$$
	and a  (measurable)  process $\Gamma'$ such that one has strict inequality ``$>$'' if $\Gamma$ is replaced by $\Gamma'$ in the above formula.
	Therefore, for the optimal control $\Ph$, one has
	$$
		H_s(\cdot,Z_s,\Gamma_s, \widehat m_s)
		=
		\widehat \sigma_s \beta^{\Ph}_s \cdot Z_s 
		+
		\frac12 \widehat \sigma^2_s :\Gamma_s
		-
		c_s\big(\widehat \sigma_s \beta^{\Ph}_s, \widehat \sigma^2_s,\widehat{m}_s\big),
		\quad \mathrm{Leb}\otimes\Ph\mbox{-}a.e.
	$$
	By standard convex duality, one has 
	$$ 
	   \left(\widehat \sigma_s \beta^{\Ph}_s, \frac12\widehat\sigma_s^2\right) \in \partial_{(z,\gamma)} H_s(\cdot,Z_t,\Gamma_t) \subseteq U, \quad \mathrm{Leb}\otimes\Ph\mbox{-}a.e. 
	$$  
 	We then define $\overline{\Gamma}_s:=\Gamma_s\1_{\{K_s=0\}} + \Gamma'_s \1_{\{K_s > 0\}}$. 
 	Consequently, $\Gamma=\overline{\Gamma}$, $\mathrm{Leb}\otimes\Ph$-a.e., and $(Z, \overline{\Gamma},\Ph) \in \mbox{MKV}(\mu_0, \mu_1)$, and the random variable $\xi$ has the required decomposition  $\xi=V_0^+ + Y_1^{Z, \overline{\Gamma}, \widehat\P}$, $\Ph$-a.s.

	\vspace{3mm}

	\noindent (ii-4) Finally, the equality in \eqref{eq:equality_optimal_ctrl} follows immediately from the last construction of $\overline{\Gamma}$ as a probability measure $\P$ is optimal if and only if the corresponding process $K^\P = 0$, $\P$-a.s.
\qed	
    \end{proof}

\subsection{Existence of solutions to the planning problem}

	The first inclusion in Theorem \ref{thm:2MFP_main1} provides a systematic way to construct solutions to the mean field planning problem with given marginals $\mu_0$ and $\mu_1$,
	that is, it is enough to construct a McKean-Vlasov dynamics $X$ in \eqref{eq:MKV_SDE}  with given marginal distributions.
	Under further conditions, our next result ensures that it is enough to construct a semi-martingale measure in $\Pc^U(\mu_0, \mu_1)$.

%
	
	\begin{Proposition} \label{prop:dH_full_image}
		$\mathrm{(i)}$ In both settings \eqref{eq:def_PcU_0} and \eqref{eq:def_PcU_1} for the definition of $\Pc^U(\mu_0)$, assume that, for all $(b, \frac12 a) \in U$ and $(s,\om, m) \in [0,1] \x \Om \x \Pc(\R^d)$,
		there exits a maximizer for $\sup_{(z, \gamma) \in D_H(\cdot)} \big\{ b \cdot z + \frac12 a : \gamma - H(\cdot, z, \gamma)\big\}$.
		Then, for all $\Ph \in \Pc^U(\mu_0, \mu_1)$, there exists a measurable selection
		$(\bar b,  \frac12 \bar \sigma^2)(\cdot)$ in $\partial_{(z,\gamma)} H_s(\cdot) \subseteq U$,
		together with $\F$-progressively measurable processes $(Z, \Gamma) \in D_m(\cdot)$ such that
		$$
			X_t = X_0 + \int_0^t \overline b_s \big(Z_s, \Gamma_s, \Ph \circ X_s^{-1} \big) ds 
			+ \int_0^t \overline\sigma_s \big(Z_s, \Gamma_s, \Ph \circ X_s^{-1}  \big) dW^{\Ph}_s,
			~~\Ph\mbox{-a.s.},
		$$
		where $W^{\Ph}$ is $\Ph$-Brownian motion, i.e.~$(Z, \Gamma,\Ph)\in \mathrm{MKV}_0(\mu_0,\mu_1)$.
		
		\vspace{0.5em}
		
		\noindent $\mathrm{(ii)}$ Consequently, if $Z\in\Hc^2(\mu_0)$, then we obtain that $(Z, \Gamma,\Ph)\in \mathrm{MKV}(\mu_0,\mu_1)$ and hence $Y^{Z,\Gamma,\widehat{\P}}_1\in\mathrm{MFP}(\mu_0,\mu_1)$.
	\end{Proposition}
	\begin{proof}
	We notice that the Hamiltonian $H(\cdot, z, \gamma)$ is convex in $(z, \gamma)$,
	then by standard result in convex analysis, see, e.g., \cite[Theorem 23.5]{Rockafellar1970}, $(\widehat z, \widehat \gamma)$ is maximizer of  
	$$ \sup_{(z, \gamma) \in D_H(\cdot)} \bigg\{ b \cdot z + \frac12 a : \gamma - H(\cdot, z, \gamma)\bigg\} =: c^{**}(\cdot, b, a) $$ 
	if and only if 
	 \begin{align*}
	 \Big(\widehat z, \frac12\widehat\gamma\Big) \in \partial_{(b,a)} c^{**}(\cdot,b,a),
	 \end{align*}
	which is is equivalent to 
	\begin{align*}
		\Big(b, \frac12 a\Big) \in \partial_{(z,\gamma)} H\big(\cdot, \widehat z, \widehat \gamma\big).
	\end{align*}
	Let $\Ph \in \Pc^U(\mu_0, \mu_1)$ so that the dynamic of the canonical process $X$ is given by
	$$
		dX_t = b^{\Ph}_t dt + \widehat \sigma_t dW^{\Ph}_t,~~ \Ph\mbox{-a.s.}
		~~~\mbox{and}~~~
		\big(b^{\Ph}, \widehat \sigma\big) \in U,~~ \mathrm{Leb}\otimes\Ph\mbox{-}a.e.
	$$
	Let $\widehat m_s := \Ph \circ X^{-1}_s$, $s \in [0,1]$. 
	One can use the measurable selection theorem to choose a version of sub-gradient $\big(\overline b,  \frac12\overline \sigma^2\big)(\cdot)$ in $\partial_{(z,\gamma)} H(\cdot)$, together with $\F$-progressively measurable processes $(Z, \Gamma)$ such that
	$$
		(b^{\widehat\P}_t, \widehat \sigma^2_t) = ( \overline b, \overline \sigma^2)(Z_t, \Gamma_t, \widehat m_t), ~~\mathrm{Leb}\otimes\Ph\mbox{-}a.e.,
	$$
	and hence $(Z, \Gamma,\Ph)\in \mathrm{MKV}_0(\mu_0,\mu_1)$.
	\qed
	\end{proof}

\paragraph{Construction of martingales with given marginals}

	Let $U = \{0\} \x \S^d_+$, and 
	$$
		c_s(\om, b, a, m) 
		:=
		\frac14 C_m |a|^2,
		~~~\mbox{with}~~
		C_m := \int_{\R^d} (1+|x|^2) m(dx).
	$$
	Then, the Hamiltonian has domain $D_H(\cdot) = \R^d \x \S^d$, and
	$$
		H_s(\om, z, \gamma, m)
		 =
		\frac14 C_m^{-1} |\gamma|^2,
		~~~
		\partial_{(z,\gamma)}H(\cdot, z, \gamma,m) = \Big\{ \Big( 0,  \frac12 C_m^{-1} \gamma \Big) \Big\}.
	$$
	Clearly, all the conditions in Proposition \ref{prop:dH_full_image} hold true.
	Therefore, to find a solution to the mean field planning problem $\mathrm{MFP}(\mu_0, \mu_1)$,
	a first approach consists in finding $\Ph \in \Pc^U(\mu_0, \mu_1)$ so that
	$$
		X_t = X_0 + \int_0^t \widehat \sigma_s dW^{\Ph}_s, \quad \Ph\mbox{-a.s.},
	$$
	and hence with $Z \equiv 0$, $\Gamma_t := C_{\widehat m_t} \widehat \sigma^2_t$, $\widehat m_t := \Ph \circ X_t^{-1}$,
	one has $(\Ph, Z, \Gamma) \in \mathrm{MKV}(\mu_0, \mu_1)$.
	
	\vspace{0.5em}
	
	In this setting and when $d= 1$, 
	the problem of finding an element $\Pc^U(\mu_0, \mu_1)$ is equivalent to the so-called Skorokhod embedding problem, which consists in finding a stopping time time $\tau$ in some filtered probability space equipped with a Brownian motion $W$ such that
	$$
		W_0 \sim \mu_0, \quad W_{\tau} \sim \mu_1.
	$$
	Indeed, given $\Ph \in \Pc^U(\mu_0, \mu_1)$, under which $X$ is a diffusion martingale with marginals $\mu_0$ and $\mu_1$.
	By Dambis-Dubins-Schwarz theorem, one can represent $X$ as a time-changed Brownian motion, i.e.,
	$$
		X_t = W_{\langle X \rangle_t}, \quad t \in [0,1],
	$$
	where $W$ is some Brownian motion and $\langle X \rangle_t$ are stopping times w.r.t.~the time-changed filtration.
	Thus $(W, \langle X\rangle_1)$ provides a solution to the Skorokhod embedding problem with marginals $(\mu_0, \mu_1)$.
	Conversely, given a solution $(W, \tau)$, to the Skorokhod embedding problem with marginals $\mu_0$ and $\mu_1$,
	let us define
	$$
		X_t := W_{\tau \wedge \frac{t}{1-t}}, \quad t \in [0,1].
	$$
	Then, it is easy to check that $X$ is a martingale diffusion process such that $X_0 = W_0 \sim \mu_0$ and $X_1 = W_{\tau} \sim \mu_1$, and hence $\P \circ X^{-1} \in \Pc^U(\mu_0)$.
	
	\vspace{0.5em}
	
	We also notice that the Skorokhod embedding problem has a solution if and only if $\mu_0$ and $\mu_1$ have finite first order moment and $\int_{\R} \phi(x) \mu_0(dx) \le \int_{\R} \phi(x) \mu_1(dx)$ for all convex function $\phi$.
	Moreover, there are various constructions of solutions to the Skrokohod embedding problem, and many of them enjoy some optimal property, see e.g.~Ob\l{}\'oj \cite{Obloj2004} for a survey.
	Consequently, the induced solution of the mean field planning problem enjoys the same optimal property among all possible solutions, and hence solves the corresponding optimal planning problem as discussed in Section \ref{subsec:entropicMFG}.

\paragraph{Construction of semi-martingales with given marginals}
	
	When $U = \R^d \x \S^d$, the problem of construction a semi-martingale measure in $\Pc^U(\mu_0, \mu_1)$ is very easy. Let us report the construction by using the so called Bass solution of the Skorohod embedding problem, see e.g.~Ob\l{}\'oj \cite{Obloj2004} .
	Assume that both $\mu_0$ and $\mu_1$ have finite first order moment, 
	and let $B$ be a standard Brownian motion in a probability space $(\Om^*, \Fc^*, \P^*)$,
	together with  a random variable $\Xh_0$ independent of $B$ such that
	$\P^* \circ \Xh^{-1}_0 = \mu_0$.
	We can then find some measurable function $T: \R^d \to \R^d$ such that $\P^* \circ \big( T(B_1) \big)^{-1} \sim \mu_1$.
	Next, by martingale representation, there exists a constant $c \in \R$ and a predictable process $\sigmat$ (w.r.t.~the Brownian filtration generated by $B$) such that 
	$$
		T(B_1) = c + \int_0^1 \sigmat_t dB_t.
	$$
	We then define a process $\Xh$ by
	$$
		\Xh_t := \Xh_0 + \int_0^t \big(c - \Xh_0 \big) ds + \int_0^t \sigmat_s dB_s.
	$$
	It is immediate to check that
	$$
		\Ph ~:=~ \P^* \circ \Xh^{-1} ~\in~ \Pc^U(\mu_0, \mu_1).
	$$
	
\paragraph{Optimal transport along controlled McKean-Vlasov dynamic}

	As in the discussion in Section \ref{subsec:entropicMFG},
	one can consider an optimal mean field planning problem, by choosing an optimal solution $\xi$ in the class $\Xi(\mu_0, \mu_1)$ w.r.t.~some criteria.
	The problem can be reduced to an optimal transport problem along controlled McKean-Vlasov dynamic:
	for some reward function $\Psi$, one solves 
	$$
		\sup_{(Z,\Gamma, \Ph) \in \mathrm{MKV}(\mu_0, \mu_1) }
		\Psi\big(Z, \Gamma, \Ph\big),
	$$
	where we recall from \eqref{eq:MKV_SDE} that $\mathrm{MKV}(\mu_0, \mu_1)$ is the set of all $(Z, \Gamma, \Ph)$
	such that,
	with a version of sub-gradient  $\big(\overline b_s, \frac12 \overline \sigma_s^2\big)(z, \gamma, m) \in\partial_{(z,\gamma)} H_s( z, \gamma, m) \subseteq U$,
	$\Ph$ is weak solution to the McKean-Vlasov equation:
	$$
		X_t = X_0 + \int_0^t \overline b_s \big(Z_s, \Gamma_s, \Ph \circ X_s^{-1} \big) ds 
		                     + \int_0^t \overline\sigma_s \big(Z_s, \Gamma_s, \Ph \circ X_s^{-1}  \big) dW^{\Ph}_s,
		~~~\Ph\mbox{-a.s.},
	$$
	under the marginal constraints:
	$$
		\Ph \circ X_0^{-1} = \mu_0 ~~~\mbox{and}~~~\Ph \circ X_1^{-1} = \mu_1.
	$$
	Such a problem extends the classical optimal transport problem studied in the literature,
	such as the martingale optimal transport in Beiglb\"ock, Henry-Labord\`ere, and Penkner \cite{BHLP2013} and Galichon,
	Henry-Labord\`ere, and Touzi \cite{GHLT2014}, 
	the martingale Benamou-Brenier problem in Huesmann and
	Trevisan \cite{HT2019}, Backhoff-Veraguas, Beiglb\"ock, Huesmann, and K\"allblad \cite{BBHK2020}, or the semimartingale optimal transport problem in Mikami and Thieullen \cite{MT2008}, Tan and Touzi \cite{TT2013}, etc.

\bibliography{MFGplanning}
 \bibliographystyle{abbrv}

\end{document}